\documentclass{amsart}

\usepackage{amsfonts,amsmath, amsthm, amssymb, latexsym, epsfig}
\usepackage[english]{babel}
\usepackage[alphabetic]{amsrefs}
\usepackage[utf8]{inputenc}
\usepackage[all]{xy}
\usepackage{xspace}
\usepackage{amsmath}
\usepackage{amscd}
\usepackage{enumitem}
\usepackage{comment}
\usepackage{setspace}
\usepackage{stmaryrd}
\usepackage{xcolor}
\usepackage{tikz-cd}
\usepackage{hyperref}
\usepackage[mathscr]{eucal}
\usepackage{mathtools}

\numberwithin{equation}{section}

\usepackage[hmargin=3cm,vmargin=3cm]{geometry}

\newcommand{\C}{{\mathbb{C}}}

\newcommand{\supp}{\mathrm{supp}\,}
\newcommand{\dist}{\operatorname{dist}}

\newcommand{\ep}{{\varepsilon}}

\renewcommand{\l}{\left}
\renewcommand{\r}{\right}

\newcommand{\bC}{\mathbb{C}}

\newcommand{\bN}{\mathbb{N}}
\newcommand{\bT}{\mathbb{T}}
\newcommand{\bZ}{\mathbb{Z}}
\newcommand{\bR}{\mathbb{R}}
\newcommand{\bee}{\begin{equation}}
\newcommand{\ene}{\end{equation}}
\newcommand{\bone}{\mathbf{1}}
\newcommand{\one}{\mathrm{1}}
\newcommand{\Arg}{\operatorname{Arg}}
\newcommand{\rank}{\operatorname{rank}}

\newcommand{\ran}{\operatorname{ran}}
\newcommand{\sign}{\operatorname{sign}}

\newcommand{\ol}{\text{OL}}

\newcommand{\OL}{\operatorname{OL}}
\newcommand{\isospec}{\operatorname{isospec}}

	\newtheorem{thm}{Theorem}[section]
	\newtheorem{cor}[thm]{Corollary}
	\newtheorem{lem}[thm]{Lemma}
	\newtheorem{prp}[thm]{Proposition}

	\theoremstyle{definition}

    \newtheorem{question}[thm]{Question}

	\theoremstyle{remark}
	\newtheorem{rmk}[thm]{Remark}

\title{On almost commuting unitary matrices
}
\author{Adam Dor-On}
\address{Department of Mathematics\\University of Haifa\\199 Aba Khoushi Ave\\Mount Carmel, Haifa\\3103301, Israel}
\thanks{A. D. was partially supported by NSF / BSF grants 2350543 / 2023695 (respectively), 2452324 / 2024734 (respectively), a Horizon Marie-Curie SE project no. 101086394, and a DFG Middle-Eastern collaboration project 529300231.}
\email{adoron.math@gmail.com}
\author{Lucas Hall}
\address{Department of Mathematics\\University of Haifa\\199 Aba Khoushi Ave\\Mount Carmel, Haifa\\3103301, Israel}
\email{lucashall94@gmail.com}
\thanks{L. H. was supported by a Zuckerman Postdoctoral Fellowship in the 2023--2024 cohort.}
\author{Ilya Kachkovskiy}
\address{Department of Mathematics\\ Michigan State University\\
Wells Hall, 619 Red Cedar Rd\\ East Lansing, MI\\ 48824\\ USA}
\email{ikachkov@msu.edu}
\thanks{I. K. was partially supported by the NSF grants DMS--1846114, DMS--2052519, and the 2022 Sloan Research Fellowship.}
\date{}

\subjclass[2020]{47A05, 47L30, 15A27.}

\begin{document}

\maketitle

\begin{abstract}
A question going back to Halmos asks when two approximately commuting matrices of a certain kind are close to exactly commuting matrices of the same kind. It has long been known that there is an obstruction for approximately commuting unitary matrices to be close, in a dimension-independent way, to genuinely commuting unitary matrices. In this paper, under the vanishing of the said obstruction, we obtain effective bounds for the distance to commuting unitary matrices in terms of the commutator of the original matrices.
\end{abstract}

\section{Introduction}
Let $H$ be a (complex) Hilbert space. In 1976, P. Halmos stated the following question in \cite{Halmos}:
\begin{question}\label{halmos_question}
Is it true that for every $\varepsilon>0$ there is a $\delta>0$ such that, if $x, y\in B(H)$ satisfy 
\bee
\label{eq_almost_commuting}
\|x\|\leq 1 \quad \|y\|\leq 1 \quad \|[x,y]\| = \|xy-yx\|<\delta,
\ene
then there are commuting operators $x', y'\in B(H)$ satisfying 
\bee
\label{eq_near_commuting}
\|x-x'\|+\|y-y'\|<\varepsilon?
\ene
\end{question}
It is common to refer to equation \eqref{eq_almost_commuting} as $x$ and $y$ being {\it almost commuting}, and equation \eqref{eq_near_commuting} as being {\it nearly commuting} (cf. \cite{DS}), with the question itself stated as {\it are almost commuting matrices nearly commuting?}

As already pointed out in \cite{Halmos}, this question has a negative answer in infinite dimensions \cite{Berg}. However, one can produce many non-trivial versions of it by imposing various natural restrictions on $x$ and $y$, even if $\dim H<+\infty$. In particular, a large body of work focuses on the variants of Question \ref{halmos_question} with the following conventions:
\begin{itemize}
    \item $\dim H<+\infty$, so that $x$ and $y$ can be considered to be finite matrices. However, $\delta$ is allowed to depend on $\varepsilon$ but not on the dimension of $H$ (``dimension-independent results'').
    \item $\|\cdot\|$ is the usual operator norm.
    \item The matrices $x$ and $y$ belong to a certain class. For example, as originally suggested in \cite{Halmos}, one can consider Hermitian matrices $x=x^*$, $y=y^*$.
\end{itemize}
By considering $a=x+iy$, one can restate the self-adjoint version of the above question as whether the distance from $a$ to the set $\mathcal N$ of normal matrices can be estimated in terms of $\|[a,a^*]\|$ (that is, the norm of the self-commutator of $a$). In this form, a positive answer to the question was provided in \cite{Lin_main}; see also \cite{FR} for a shorter proof. In both cases, the proof starts by assuming towards contradiction, and is reduced to a lifting problem in a certain asymptotic matrix sequence $C^*$-algebra. Extracting an explicit dependence relation between $\varepsilon$ and $\delta$ from this kind of proof then becomes a challenging reverse engineering problem, and it is unclear whether the bounds obtained in this way are a optimal. An alternative approach to the self-adjoint problem was proposed in \cite{KS}, and has led to an order sharp estimate $\dist(a,\mathcal N)\le C\|[a,a^*]\|^{1/2}$, corresponding to $\delta=C\varepsilon^2$. See also \cite{H_orig,Herrera} regarding earlier results and \cite{DS} for a survey.

\subsection{Almost commuting unitary matrices: overview and the main result} \hfill\newline 

We will be mainly interested in Question \ref{halmos_question} for a pair of unitary matrices.
In this case, the problem presents a topological obstruction, see \cite{Voiculescu,Exel,Exel_Loring}, as well as the paragraph following \cite[Question I.5]{DS}. This obstruction has been characterized in several somewhat equivalent ways, one of the most simple of them being the so-called {\it winding number invariant}. For $u,v\in {\mathrm U}(n)$ with $\|[u,v]\|<2$, we denote by $w(u,v)$ the winding number (with respect to the origin) of the curve
\bee
\label{eq_winding_1}
\omega : [0,1] \rightarrow \mathbb{C}\setminus\{0\},\quad \omega(t):= \det (t\cdot uv + (1-t)\cdot vu).
\ene
Assuming that $\varepsilon$ is small enough, it is shown in the above references that any pair $u,v\in {\mathrm U}(n)$ that is $\varepsilon$-close to a commuting pair must have $w(u,v)=0$. The winding number is indeed an obstruction due to the example of {\it Voiculescu's unitaries} \cite{Voiculescu}, which is a sequence of pairs of unitary matrices $u_n,v_n\in {\mathrm U}(n)$ satisfying $\|[u_n,v_n]\|=O(1/n)$ while $w(u_n,v_n)=-1$ \cite{Exel_Loring}, thereby showing that $u_n,v_n$ are not eventually close to a commuting pair of matrices.

In view of the discussion above, a natural setting for Question \ref{halmos_question} for two unitaries includes the additional assumption that their winding number vanishes. In this stated form, \cite{Eilers1} and \cite{Gong_Lin} independently resolved the problem at around the same time (see also \cite{LorHas}, \cite{LorSor} for applications to physics). Similarly to the self-adjoint case considered in \cite{Lin_main,FR}, both solutions involve an asymptotic sequence algebra construction, and therefore do not provide an explicit relation between $\ep$ and $\delta$, leaving the corresponding quantitative version of Question \ref{halmos_question} unresolved, as pointed out in the discussion preceding \cite[Theorem 4.10]{HL}.

The main result of the present paper answers the above question by providing quantitative bounds for the distance in terms of the norm of the commutator for unitary matrices with vanishing winding number obstruction as follows:
\begin{thm}\label{th_main}
There exists an absolute constant $C>0$ such that for every $n\in \mathbb N$ and $u,v\in {\mathrm U}(n)$ with $w(u,v)=0$ one can find a pair $u',v'\in {\mathrm U}(n)$ such that
$$
\|u-u'\|+\|v-v'\|\le C\|[u,v]\|^{2/15},\quad [u',v']=0.
$$
\end{thm}

We next provide an overview of the proof. We start by noting that there are various versions of the obstruction to proximity to a nearly commuting pair, with different potential directions of generalizations; however, they all coincide in our setting of pairs of almost commuting unitary matrices. We will use the winding number $w(u,v)$ of the curve defined by equation \eqref{eq_winding_1} and the {\it isospectral invariant} $\isospec(u,v)$, originally introduced in \cite[Theorem 4.1]{BEEK}, see equation \eqref{eq_isospec_def} in Section 3. For the convenience of the reader, we show that $w(u,v)=\isospec(u,v)$ in Corollary \ref{cor_isospec_winding}, since the original proof is spread among several of the above-referenced papers. The proof of Theorem \ref{th_main} is a combination of the following two results. The first one is a quantitative version of the {\it isospectral homotopy lemma} from \cite{BEEK}.

\begin{lem}
\label{lemma_homotopy_intro}
There exists an absolute constant $C>0$ such that for every two unitary matrices $u,v\in {\mathrm U}(n)$ with $\isospec(u,v)=0$, there exists a piecewise smooth path $v_t\colon [0,1]\to {\mathrm U}(n)$ of fixed length with
\bee
\label{eq_path}
v_0=v,\quad v_1=\bone,\quad \|[v_t,u]\|\le C\|[u,v]\|^{2/5}.
\ene
\end{lem}

The second result is the following quantitative variant of a theorem, whose analogue in the setting of real rank zero $C^*$-algebras was proved in \cite{Lin_exprank} and \cite{Phillips} independently and at around the same time.
\begin{thm}
\label{th_gap_opening_intro}
Suppose that $u, v$ are unitary matrices such that $\|[u,v]\|\le \delta$ and that there exists a continuous path $\{u_t\colon t\in [0,1]\}$ such that:
$$
u_0=u,\quad  u_1=\bone,\quad \|[u_t,v]\|\le \delta \quad \forall t\in [0,1],
$$
Then there exist unitary matrices $u'$, $v'$ such that
$$
[u',v']=0,\quad \|u-u'\|+\|v-v'\|\le C\delta^{1/3}
$$
\end{thm}
{\noindent \bf Proof of Theorem $\ref{th_main}$.} Since $w(u,v)=0$, Corollary \ref{cor_isospec_winding} implies $\isospec(u,v)=0$. Then, the result follows from Theorem \ref{th_gap_opening_intro} applied with $\delta=C\|[u,v]\|^{2/5}$, using the conclusion of Lemma \ref{lemma_homotopy_intro}.\,\,\qed \hfill\newline

The proof of Lemma \ref{lemma_homotopy_intro}  is obtained by retracing and providing quantitative bounds in several steps in \cite{BEEK}. While the direct repetition of the arguments does not lead to the exponent $2/5$, we make certain optimizations along the way to acheive this (see Subsection \ref{homotopy_lemma_proof}).

Theorem \ref{th_gap_opening_intro} presents a different kind of challenge, and, while it does not technically involve $C^*$-algebras, the corresponding language of $C^*$-algebras best explains our motivation. We first note that, if one of the unitaries has a large gap in its spectrum, then the problem becomes ``topologically trivial'' and essentially reduces to the main result of \cite{KS} (see Proposition \ref{prop_top_trivial}, Lemma \ref{lemma_almost_commuting_gapped}, and Remark \ref{remark_macro}). Our construction of commuting approximants, which involves several applications of the above ``gap-opening'' result, involves a sequence of operator-theoretic procedures, mainly applied to the matrix $u$, where at each step one needs to maintain control over the norm of the commutator with $v$. Not every operator-theoretic procedure has this property. For example, considering spectral projections or discontinuous functional calculus may blow up the norm of the commutator.

Difficulties of a similar kind appeared in the proof of the (first) Brown--Pedersen Conjecture in \cite{Lin_exprank} and \cite{Phillips}, which asks whether the multiplier algebra of a real-rank zero $C^*$-algebra remains of real-rank zero. In that context, the process could also be described as gap-opening by a sequence of operator-theoretic procedures, to be performed while remaining inside of the given $C^*$-algebra.

Our main observation, which was also used in \cite{KS}, is that many procedures that allow us to remain in a given $C^*$-algebra also have ``commutator control'' properties. For example, smooth functional calculus for normal elements has this property due to theory of operator Lipschitz functions (see (OL4) and (OL6) in Appendix \ref{appendix-A}). Therefore, the aspects of certain $C^*$-algebraic proofs can be transferred into proofs that involve almost commuting matrices with commutator control. The proof in Subsection \ref{ss:amplified-pair} is the result of such a transfer of the construction in \cite{Lin_exprank}. While smooth functional calculus may be considered straightforward, we draw the reader's attention to the following three more involved aspects of our work.
\begin{enumerate}
    \item The main result of \cite{Lin_exprank} states that one can open a gap in a unitary element of a real rank zero $C^*$-algebra, provided that it belongs to the connected component of the identity in the unitary group. In Lemma \ref{lemma_homotopy_intro} we provide a quantitative refinement of the isospectral homotopy lemma from \cite[Lemma 6.1]{BEEK}, which connects one unitary to the identity with controlled losses to commutator control along the path. The  transfer from path connectedness to the identity into commutator control properties suggests a quantitative refinement of \cite{Lin_exprank}, which we realize in Theorem \ref{th_gap_opening_intro}.
    \item The real-rank zero property for a given $C^*$-algebra $\mathcal{A}$ is equivalently stated as requiring that every self-adjoint element $a$ in $\mathcal A$ is close to another self-adjoint element $a'$ in $\mathcal A$ whose spectral projections remain in the $C^*$-algebra $\mathcal{A}$. Item (4) of Proposition \ref{prop_top_trivial} captures the commutator analogue of this property. Indeed, if $a=a^*$ and $[a,v]$ is small, one can apply item (4) of Proposition \ref{prop_top_trivial} once, and produce an exactly commuting pair $[a',v']=0$. Afterwards, if $f$ is a bounded Borel function, then $[f(a'),v']=0$. Since $v$ is close to $v'$, the commutator $[f(a'),v]$ will therefore remain small.
    \item Most of the constructions from \cite{Lin_exprank} happen in the amplification of the original $C^*$-algebra, and it is an important fact that any amplification of a real rank zero algebra is still of real rank zero. The commutator analogue of this, considered in the previous paragraph, is also preserved upon amplification, but with an increasingly large losses in the estimates as the size of the amplifications increases. In that approach, the estimates would also depend on the length of the path obtained in Lemma \ref{lemma_homotopy_intro}. Instead of this, the strategy in our paper is to use \cite{KS} in the amplified matrix algebra directly, thereby avoiding an (exponential) loss in the final distance bounds.
\end{enumerate}
We note that each ``transfer'' of an abstract $C^*$-algebraic construction into the commutator-controlled language leads to some loss in the exponent of the commutator norm in the main result, and different interpretations of these constructions can lead to vastly different quantitative results. Controlling these losses was one of the main technical challenges in our work.

\subsection{Structure of the paper}\hfill\newline
\nopagebreak

In Section \ref{s:stl}, we state some (mostly known) results on spectral projections and their behavior under small perturbations. In Subsection \ref{ss:spect-proj-pert}, we discuss various relations between spectral projections associated to pairs of intervals with disjoint endpoints. In Subsection \ref{ss:orth-proj-unitary}, we include some general properties of projections and unitary operators. Subsection \ref{ss:unitary-equiv-fam-proj} deals with a quantitative version of the result on intertwined families of projections, whose qualitative analogues were obtained in \cite{BEEK}. 

In Section \ref{s:isospec-homot-wind-numb}, we discuss our topological invariants for pairs of unitary matrices. In Subsection \ref{ss:isospec-def}, we discuss a simplified version of the isospectral invariant. In Subsection \ref{homotopy_lemma_proof}, we prove Lemma \ref{lemma_homotopy_intro}m which is the isospectral homotopy lemma with commutator control. In \ref{ss:equiv-bet-inv}, our isospectral homotopy lemma is used to establish an equivalence between the winding number and isospectral invariants. In Subsection \ref{ss:related}, we discuss an equivalent definition of the winding number invariant from \cite{Exel}, as well as situations where there is no obstruction to proximity to a commuting unitary pair in infinite dimensions. In Section \ref{s:commuting-approx} we prove Theorem \ref{th_gap_opening_intro}. Subsection \ref{ss:prep} contains some preparations, including Lemma \ref{lemma_almost_commuting_gapped} on almost commuting unitary matrices with one having a spectral gap. Following and transferring the general ideas of \cite{Lin_exprank} to the commutator-controlled setting, we use Proposition \ref{prop_unitary_double} to construct an amplification of the original almost commuting pair and perturb it in the larger space to create a spectral gap in Subsection \ref{ss:amplified-pair}. This allows to find commuting approximants in a larger space in Subsection \ref{ss:amplified-comm}. Afterwards, we descend these approximants back into the original space, which is done in two dimension-reduction steps in Subsections \ref{ss:firstreduction} and \ref{ss:secondreduction}.

\section{Preliminaries} \label{s:stl}
In this section, we gather several (mostly well known) results on spectral projections associated to isolated intervals in the spectra of unitary matrices, and their behavior under small perturbations. The proofs are included mostly for the sake of completeness, as well as to introduce notation that will be useful in the later sections.

We will actively use the results from the theory of {\it operator Lipschitz functions}, often citing properties (OL1) -- (OL8) stated in Appendix \ref{appendix-A}. Let $u_1,u_2\in {\mathrm U}(n)$ and $g\colon \mathbb T=\mathbb R/\mathbb Z\to \mathbb R$ be a smooth function on the circle. Then $\|g(u_1)-g(u_2)\|$ is small provided that $\|u_1-u_2\|$ is small. As a consequence, if $f$ is another bounded function whose support is disjoint from that of $g$, we have $f(u_1)g(u_1)=0$ and therefore $\|f(u_1)g(u_2)\|$ is small.

If $g$ is, say, discontinuous, then the first claim from the previous paragraph falls apart. However, some statements of this kind remain true if the supports of $f$ and $g$ are sufficiently disjoint. Most of the results below have been used in previous work such as \cite{BEEK,Eff}, in many cases with better constants. The use of operator Lipschitz functions allows us to provide the proofs in a somewhat unified language. We begin with an elementary fact about bump-type functions.
\begin{prp}
\label{prop_bump}
Let $\beta>0$ and $I_1,\ldots,I_n\subset \mathbb T$ be a system of disjoint intervals on the circle with $\mathrm{dist}(I_k,I_\ell)\ge \beta$ for $k\neq\ell$. For $a_1,\ldots,a_n\in [-1,1]$ there exists
$$
f\in C^{\infty}(\mathbb T;\mathbb R),\quad |f|\le 1,\quad \left.f\right|_{I_j}=a_j,\quad \|f\|_{\OL(\mathbb T)}\le \frac{C}{\beta},
$$
where $C$ is an absolute constant.
\end{prp}
\begin{proof}
From (OL7) in Appendix \ref{appendix-A}, the problem can be restated in the language of $1$-periodic function on $\mathbb R$. By rescaling (OL5), at the expense of the factor $1/\beta$ one can consider $\beta=1$. In this case, one can construct $f$ as a linear combination of bump functions with disjoint supports which implies a uniform bound in $C_b^2(\mathbb R)$ and allows to complete the proof by applying (OL3).
\end{proof}
\begin{rmk}
Note that rescaling must be applied before (OL3), since application to $\|f\|_{C_b^2(\mathbb R)}$ directly results in a factor $\beta^{-2}$ instead of $\beta^{-1}$.
\end{rmk}

\subsection{Behavior of spectral projections under small perturbations} \label{ss:spect-proj-pert}\hfill\newline

The following two lemmas, used in \cite[Lemma 2.2 \& Lemma 2.9]{BEEK} and other related work, formalize the principle that the orthogonality relations between spectral projections (and, more generally, functions with disjoint supports) are approximately preserved under small perturbations, provided that the boundaries of the intervals under consideration are sufficiently separated from one another.
\begin{lem}
\label{lemma_spectral1}Let $g_j\colon\bT\to [-1,1]$ be measurable and $u_j\in {\mathrm U}(n)$, $j=1,2$.
\begin{itemize}
	\item[$(i)$] Suppose that $I,J\subset \bT$ are disjoint closed intervals and $\supp g_1\subset I$, $\supp g_2\subset J$. Then
$$
\|g_1(u_1)g_2(u_2)\|\le \frac{C\|u_1-u_2\|}{\dist(I,J)}.
$$
\item[$(ii)$] Suppose that $J\subset I\subset \bT$ are intervals and $\supp g_1\subset J$, $\supp (1-g_2)\subset \bT\setminus I$. Then
$$
\|g_2(u_2)g_1(u_1)-g_1(u_1)\|\le \frac{C\|u_1-u_2\|}{\dist(J,\bT\setminus I)}.
$$
\end{itemize}
\end{lem}
\begin{proof}
In the setting of (i), let $h_1, h_2\colon \bT\to [0,1]$ be smooth functions such that 
$$
h_1|_I\equiv 1, \quad h_2|_J\equiv 0;\quad h_1|_J\equiv 0, \quad h_2|_I\equiv 1,\quad \supp h_1\cap \supp h_2=\emptyset.
$$
From Proposition \ref{prop_bump}, one can choose these functions in a way that $\|h_j\|_{\OL(\bT)}\le \frac{C}{\dist(I,J)}$, where $C$ is an absolute constant. Then
\begin{multline*}
\|g_1(u_1)g_2(u_2)\|=\|g_1(u_1)h_1(u_1)h_2(u_2)g_2(u_2)\| \le \|h_1(u_1)h_2(u_2)\|\\ \le \|h_1(u_1)h_2(u_1)\|+\|h_1(u_1)\left(h_2(u_2)-h_2(u_1)\right)\|
\le \|h_2(u_2)-h_2(u_1)\|\\  \le \|h_2\|_{\OL(\bT)}\|u_1-u_2\|\le \frac{C\|u_1-u_2\|}{\dist(I, J)},
\end{multline*}
which implies part (i). Here, we used the fact that $h_1(u)h_2(u)=(h_1 h_2)(u)=0$ and the definition of $\|\cdot\|_{\OL}$ (see Appendix \ref{appendix-A}). Part (ii) follows from Part (i) applied to the functions $1-g_2$ and $g_1$.
\end{proof}
We denote by $\one_J$ the indicator function of a set $J\subset \bT$, and by $\bone$ the identity operator. 
\begin{lem}
\label{lemma_spectral2}Let $I,J\subset \bT$ be two intervals with disjoint boundaries, and $u_j\in {\mathrm U}(n)$, $j=1,2$. Then
$$
\|[\one_I(u_1),\one_J(u_2)]\|\le \frac{C\|u_1-u_2\|}{\dist(\partial I,\partial J)},
$$
where $C>0$ is an absolute constant.
\end{lem}
\begin{proof}
Since $\one_{I}(u_j)=\bone-\one_{\bT\setminus I}(u_j)$, one can replace $I$ by $\bT\setminus I$ and/or $J$ by $\bT\setminus J$ without loss of generality. As a consequence, the cases where $I$ and $J$ are disjoint, when $I$ and $J$ have disjoint complements, or when one interval is contained in the other follow from Lemma \ref{lemma_spectral1} (i) (both terms in the commutator can be estimated separately). It remains to consider the case where exactly one endpoint $x$ of $J$ is in the interior of $I$. Let $\beta:=\dist (\partial I,\partial J)$, and consider a smooth bump function $h\colon \bT\to [0,1]$ such that $h=1$ in the $\beta/10$-neighborhood of $x$ and $h=0$ outside of the $\beta/5$-neighborhood of $x$. From Proposition \ref{prop_bump}, one can choose $h$ such that $\|h\|_{\OL(\bT)}\le C\beta^{-1}$.

By construction, we have $h\le \one_I$, while the function $\one_I-h$ vanishes in a neighborhood of $x$. The set $\supp(\one_I-h)$ is a union of two intervals which we can denote by $I_1,I_2$. Setting $h_k:=( 1_I-h)|_{I_k}$ for $k=1, 2$, we have
$$
\one_I=h_1+h+h_2,\quad \supp h_j\subset I_j,\quad h_j\colon \mathbb T\to [0,1].
$$
One of the intervals $I_j$ (say, $I_1$) must be contained in $J$, and the other (say, $I_2$) disjoint from $J$. In both cases, we have $\dist(\partial I_j,\partial J)\ge \beta/5$. As a consequence,
\bee
\label{eq_commutator1}
\|[h_1(u_1),\one_J(u_2)]\|=\|[h_1(u_1),\one_{\bT\setminus J}(u_2)]\|\le 2\|h_1(u_1)\one_{\bT\setminus J}(u_2)\|\le \frac{10C\|u_1-u_2\|}{\beta}
\ene
by Lemma \ref{lemma_spectral1} (i), since $\dist(I_1,\bT\setminus J)\ge \beta/5$. As $\dist(I_2,\bT\setminus J)\ge \beta/5$, we have also
\bee
\label{eq_commutator2}
\|[h_2(u_1),\one_J(u_2)]\|\le 2\|h_2(u_1)\one_{\bT\setminus J}(u_2)\|\le \frac{10C\|u_1-u_2\|}{\beta}.
\ene
Finally, using the smoothness of $h$, we have
\bee
\label{eq_commutator3}
\|[h(u_1),\one_J(u_2)]\|\le \|[h(u_1)-h(u_2),\one_J(u_2)\|+\|[h(u_2),\one_J(u_2)]\|\le 2\|h\|_{\OL(\bT)}\|u_1-u_2\|\le \frac{2C \|u_1-u_2\|}{\beta}.
\ene
Combining the equations \eqref{eq_commutator1} -- \eqref{eq_commutator3} completes the proof.
\end{proof}

\subsection{Perturbations of orthogonal projections and unitary operators} \label{ss:orth-proj-unitary}\hfill\newline

We now collect several (mostly well-known) results about orthogonal projections, operators close to projections, almost commuting projections, and unitary operators. Recall that every bounded invertible operator $a$ on a Hilbert space admits a polar decomposition
$$
a=u|a|,
$$
where $u$ is unitary and $|a|=(a^*a)^{1/2}$. In this situation, we will use the notation $\Arg a:=u=a|a|^{-1}$. The following fact is elementary and we omit the proof.
\begin{prp}
\label{prop_almost_projection}
Let $t_0$ be a self-adjoint operator on a Hilbert space such that $\|t_0^2-t_0\|<1/10$. Then the operator
$$
t:=\frac12(\bone+\Arg(2t_0-\bone))
$$
is an orthogonal projection satisfying
$$
\|t-t_0\|\le 4\|t_0^2-t_0\|,\quad \ran t\subset \ran t_0.
$$
\end{prp}
\begin{rmk}
Note that the assumptions on $t_0$ imply that $\sigma(t_0)\subset [-1/5,1/5]\cup [1-1/5,1+1/5]$. As a consequence, one can also define $t=\one_{[1-1/5,1+1/5]}(t_0)$.
\end{rmk}
\begin{rmk}
\label{rem_rank}
For an approximate projection operator $t_0$ satisfying the assumptions of Proposition \ref{prop_almost_projection}, define 
$$
\rank_+(t_0):=\rank(t)=\dim\ran \one_{[1-1/5,1+1/5]}(t_0)
$$
the dimension of the range of the associated exact projection. Note that $\rank_+$ is stable under small perturbations that preserve spectral gap of $t_0$ around $1/2$. As a consequence, $\rank_+$ is constant on any continuous family of approximate projections satisfying the assumptions of Proposition \ref{prop_almost_projection}.

Suppose also that $\|t_0^2-t_0\|<1/50$, $\|s_0^2-s_0\|<1/50$, $\|t_0 s_0\|<1/50$. Then $t_0+s_0$ satisfies the assumptions of Proposition \ref{prop_almost_projection} and
\bee
\label{eq_rankplus_additive}
\rank_+(t_0+s_0)=\rank_+(t_0)+\rank_+(s_0).
\ene
\end{rmk}

We will also need the following version of a well known result about pairs of projections, see \cite[Section I.4.6]{Kato}.
\begin{prp}
\label{prop_two_projections}
Let $p$ and $q$ be two orthogonal projections on a Hilbert space, with $\|p-q\|<1$. Define
$$
\sigma:=\left(qp+(\bone-q)(\bone-p)\right)\left(\bone-(p-q)^2\right)^{1/2}.
$$
Then, $\sigma$ is a unitary operator satisfying
$$
q=\sigma p\sigma^{-1},\quad p=\sigma^{-1}q\sigma.
$$
Moreover, if $\|p-q\|\le 1/10$, then
$$
\|\sigma-\bone\|\le 4\|p-q\|.
$$
\end{prp}
The last estimate can be easily obtained (if needed, with better constants) from the power series expansion of the square root. We will also use the following elementary corollary -- again, without optimizing the constants.
\begin{cor}
\label{cor_two_projections}
Let $p$ and $q$ be two orthogonal projections on a Hilbert space, with $\|pq\|<1/100$. Then, there exists a unitary operator $\sigma$ such that 
$$
\sigma p\sigma^{-1}\le \bone-q,\quad \|\sigma-\bone\|\le {5} \|pq\|.
$$
\end{cor}
\begin{proof}
Let $t_0:=(\bone-q)p(\bone-q)$. It is easy to see that $t_0$ satisfies the assumptions of Proposition \ref{prop_almost_projection}. Denote by $t$ the resulting projection. Note also that the smallness of $\|pq\|$ guarantees that $t_0$ (and, therefore, $t$) is close to $p$. Therefore, one can apply Proposition \ref{prop_two_projections} to $p$ and $t$.
\end{proof}
\begin{rmk}
\label{rem_unitary}
In Proposition \ref{prop_almost_projection}, Proposition \ref{prop_two_projections}, and Corollary \ref{cor_two_projections}, the unitary $\sigma$ is constructed by only using algebraic operations and square roots. As a consequence, both statements hold in any unital $C^*$-algebra.
\end{rmk}

If $q$ is a projection and $a$ is an operator, then we will use the following notation for the corresponding compression:
\bee
\label{eq_compression_notation}
qa|_q:=(qaq)|_{\ran q}.
\ene
Proposition \ref{prop_almost_unitary} and Lemma \ref{lemma_quadratic} below formalize the following claims: an almost unitary operator is close to a unitary operator, compression of a unitary operator by an almost commuting projection is close to a unitary operator, large spectral gaps are stable under such compressions, and compressions of commuting unitary operators almost commute, with a quadratic commutator estimate. In what follows, $B(z,R)$ denotes an open ball of radius $R$ about $z\in \C$. In all claims, $H$ is a separable Hilbert space. As always, we did not intend to optimize the constants.
\begin{prp}
\label{prop_almost_unitary}
\begin{enumerate}
    \item[$(i)$] Suppose that $w\in B(H)$ and $$\|w^*w-\bone\|\le \rho<1/5.$$ Then $\hat w=w|w^*w|^{-1/2}$ satisfies $\|w-\hat w\|\le 5\rho$ and $\hat w^*\hat w=\bone$. In particular, if $\dim H<+\infty$, then $w$ is close to a unitary operator.
    
    \item[$(ii)$] Suppose that $u\in U(H)$, $\dim H<+\infty$, and $p$ is an orthogonal projection with $\|[p,u]\|\le \eta<1/10$. Then there exists a unitary operator $\hat u$ acting on $\ran p$ such that $\|pu|_{p}-\hat u\|\le 10 \eta^2$. Moreover, if $B(z,R)\cap \sigma(u)= \varnothing$, then $B(z,R-20\eta^2)\cap \sigma(\hat u)= \varnothing$.
\end{enumerate}
\end{prp}
\begin{proof}
The first claim follows from the spectral mapping theorem applied to $w^*w$. In order to obtain the second claim, let $a:=pu|_p$. We have
$$
a^*a=pu^*pup|_p=\one_{\ran p}-pu^*(1-p)up=\one_{\ran p}-((1-p)up)^*((1-p)up).
$$
As a consequence,
$$
\|a^*a-\one_{\ran p}\|=\|(1-p)up\|^2=\|[u,p]p\|^2\le \eta^2,
$$
and then the existence of $\hat u$ follows from the first claim applied to $w=a$.
For the spectral gap estimate, note that for $x\in \ran p$ one has
$$
\|(\hat u-z)x\|\ge \|(pup-z)x\|-10\eta^2\|x\|\ge\|(u-z)x\|-\|[p,u]x\|-10\eta^2\|x\|\ge (R-20\eta^2)\|x\|.\,\qedhere
$$
\end{proof}
\begin{lem}
\label{lemma_quadratic}
\begin{enumerate}
	\item[$(i)$] Suppose that $u,v\in U(H)$ are commuting unitaries, each satisfying the assumptions of Claim $(ii)$ of Proposition $\ref{prop_almost_unitary}$ with the same projection $p$. Then the corresponding unitaries $\hat u$, $\hat v$, produced by the said proposition, satisfy
$$
\l\|[\hat u,\hat v]\r\|\le 50\eta^2.
$$
\item[$(ii)$]In addition to the above, suppose that both $u$ and $v$ commute with an orthogonal projection $e\le p$. Then the corresponding $\hat u$ and $\hat v$, in addition to the above properties, can be chosen to also commute with $e$.
\item[$(iii)$]Suppose that, in addition to the previous claim, either of the  operators $eu|_e$, $ev|_e$, $(p-e)u|_{(p-e)}$, or $(p-e)u|_{(p-e)}$ satisfies the spectral gap condition from Claim $(ii)$ of Proposition $\ref{prop_almost_unitary}$. Then the same holds for the corresponding block of $\hat u$ or $\hat v$.
\end{enumerate}
\end{lem}
\begin{proof}
For the first claim, similarly to the proof of Proposition \ref{prop_almost_unitary}, let
$$
a:=pu|_p,\quad b:=pv|_p.
$$
Then, one can easily check
$$
[a,b]=pv(1-p)up-pu(1-p)vp.
$$
As a consequence,
$$
\l\|[a,b]\r\|\le \|pv(1-p)\|\|(1-p)up\|+\|pu(1-p)\|\|(1-p)vp\|\le 2\l\|[p,u]\r\|\l\|[p,v]\r\|\le 2\eta^2.
$$
On the other hand, the conclusion of (ii) implies
$$
\|a-\hat u\|\le 10\eta^2,\quad \|b-\hat v\|\le 10\eta^2.
$$
By combining the above, we arrive to
$$
\l\|[\hat u,\hat v]\r\|\le 40\eta^2+\l\|[a,b]\r\|\le 42\eta^2,
$$
which completes the proof.

In order to obtain the second claim, note that $eu|_e$ and $ev|_e$ is already a commuting pair acting on $\ran e$. On the other hand, one can easily check that $(1-e)u|_{1-e}$ and $(1-e)v|_{1-e}$ is a commuting unitary pair, acting on $\ran(1-e)$, satisfying the assumptions of Claim $(i)$ with $p$ replaced by $(p-e)|_{1-e}$. Denote by $\hat{u}_{1-e}$, $\hat{v}_{1-e}$ the result of applying that claim. Then, the operators
$$
\hat u=eu|_e\oplus \hat{u}_{1-e};\quad \hat v=ev|_e\oplus \hat{v}_{1-e}
$$
satisfy the conclusion of Claim $(ii)$.

Claim (iii) for $eu|_e$ and $ev|_e$ follows directly from the spectral theorem (even without the term $-20\eta^2$). For the remaining two blocks, it follows from Claim $(ii)$ of Proposition \ref{prop_almost_unitary}.
\end{proof}

\subsection{Unitary equivalence between intertwined families of projections} \label{ss:unitary-equiv-fam-proj} \hfill\newline

The following somewhat more specific proposition appears in \cite[Lemma 2.8]{BEEK}. We present the statement here with improved estimates (that is, $\sqrt{N}$ instead of $N$ in equation \eqref{eq_better_estimate}) and continued neglect of constant optimization. Since we will need this more explicit formulation, we include a brief argument for the convenience of the reader.

\begin{prp}\label{quantBEEK}
Let $2\le N\in \mathbb N$, $0\le \ep<1/200$, and $P_k, Q_k$ be a family of orthogonal projections on a Hilbert space satisfying for $j,k\in \bZ/N\bZ$
$$
 \sum_k P_k =\sum_k Q_k = \bone,
$$
\bee
\label{eq_quantBEEK_assumptions}
\|P_kQ_j-Q_jP_k\|<\ep,\quad \|(P_k+P_{k+1})Q_k-Q_k\|<\ep,\quad \|(Q_{k-1}+Q_k) P_k-P_k\|<\ep,\quad \forall j,k\in \bZ/N\bZ.
\ene
Then each of the projections $Q_k$ and $P_k$ admits a decomposition 
$$
P_k=p_{2k-1}'+p_{2k}',\quad Q_k=q_{2k}'+q_{2k+1}'.
$$
where $p_j'$ and $q'_j$ are orthogonal projections for $j\in \mathbb{Z} / 2N \mathbb{Z}$ with norm estimates 
\bee
\label{eq_quantgeek_property1}
\|p_j'-q_j'\|\le 200\ep.
\ene
Moreover, there is a unitary operator $W$ which satisfies 
\bee
\label{eq_better_estimate}
p'_{j}=Wq_j'W^{-1},\quad j\in \mathbb{Z}/2N\mathbb{Z},\quad \|W-\bone\|\le 100\ep\sqrt{N}.
\ene
\end{prp}
\begin{proof}
Let
$$
r_{2k-1}:=P_k Q_{k-1}P_k,\quad r_{2k}:=P_k Q_k P_k,\quad s_{2k}:=Q_k P_k Q_k,\quad s_{2k+1}:=Q_k P_{k+1} Q_k.
$$
\begin{figure}
\begin{tikzpicture}
      \filldraw[black] ({cos(0)}, {sin(0)}) circle (0pt) node[anchor=east]{{\tiny $P_1$ }};
   \draw [black,thick,domain=1:44] plot ({cos(\x-22.5)}, {sin(\x-22.5)});
    \draw [blue,thick,domain=46:89] plot ({cos(\x-22.5)}, {sin(\x-22.5)});
    \draw [blue,thick,domain=91:134] plot ({cos(\x-22.5)}, {sin(\x-22.5)});
    \draw [blue,thick,domain=136:179] plot ({cos(\x-22.5)}, {sin(\x-22.5)});
   \draw [blue,thick,domain=181:224] plot ({cos(\x-22.5)}, {sin(\x-22.5)});
    \draw [blue,thick,domain=226:269] plot ({cos(\x-22.5)}, {sin(\x-22.5)});
    \draw [blue,thick,domain=271:314] plot ({cos(\x-22.5)}, {sin(\x-22.5)});
    \draw [blue,thick,domain=316:359] plot ({cos(\x-22.5)}, {sin(\x-22.5)});
    %\draw [black, thick, domain=0:360] plot ({1.05*cos(\x)},{1.05*sin(\x)});
      \draw [black,thick,domain=1:44] plot ({1.2*cos(\x)}, {1.2*sin(\x)});
      \filldraw[black] ({1.2*cos(22.5)}, {1.2*sin(22.5)}) circle (0pt) node[anchor=west]{{\small $Q_1$ }};
    \draw [red,thick,domain=46:89] plot ({1.2*cos(\x)}, {1.2*sin(\x)});
    \draw [red,thick,domain=91:134] plot ({1.2*cos(\x)}, {1.2*sin(\x)});
    \draw [red,thick,domain=136:179] plot ({1.2*cos(\x)}, {1.2*sin(\x)});
   \draw [red,thick,domain=181:224] plot ({1.2*cos(\x)}, {1.2*sin(\x)});
    \draw [red,thick,domain=226:269] plot ({1.2*cos(\x)}, {1.2*sin(\x)});
    \draw [red,thick,domain=271:314] plot ({1.2*cos(\x)}, {1.2*sin(\x)});
    \draw [red,thick,domain=316:359] plot ({1.2*cos(\x)}, {1.2*sin(\x)});
   % \draw [black, thick, domain=0:360] plot ({1.05*cos(\x)},{1.05*sin(\x)});

             \draw [blue,thick,domain=-1:46] plot ({2+1.2*cos(\x-22.5)}, {1.2*sin(\x-22.5)});
\filldraw[black] ({2+cos(0)}, {sin(0)}) circle (0pt) node[anchor=east]{{\tiny $P_1$ }};
        \draw [blue,thick,domain=-1:46] plot ({2+cos(\x-22.5)}, {sin(\x-22.5)});
\filldraw[black] ({2.2+cos(0)}, {sin(0)}) circle (0pt) node[anchor=west]{{\tiny $P_1$ }};
        \draw [red,thick,domain=-1:46] plot ({2+1.1*cos(\x)}, {1.1*sin(\x)});
\filldraw[black] ({2.1+cos(45)}, {sin(45)}) circle (0pt) node[anchor=west]{{\small $Q_1$ }};

        \draw [black,thick,domain=2:20] plot ({2+1.7*cos(\x)}, {1.7*sin(\x)});
        \filldraw[black] ({2.7+cos(0)}, {0.3+sin(0)}) circle (0pt) node[anchor=west]{{\tiny $r_2 = P_1Q_1P_1$ }};
\end{tikzpicture} \caption{} \label{F1}
\end{figure}
We visualize these operators according to Figure \ref{F1}. The estimates of equation \eqref{eq_quantBEEK_assumptions} then imply that each of the quantities
\bee
\label{eq_quantBEEK_proof1}
\|r_j^2-r_j\|, \ \|s_j^2-s_j\|, \ \|r_{2k-1}r_{2k}\|, \ \|s_{2k}s_{2k+1}\|,  \ \|r_j-s_j\|, \ \|r_{2k-1}+r_{2k}-P_k\|, \ \|s_{2k}+s_{2k+1}-Q_k\|
\ene
does not exceed $4\ep$ for $k\in \bZ/N\bZ$, $j\in  \bZ/2N\bZ$. In other words, the families $r_j$ and $s_j$ can be considered as approximate analogues of the desired $p',q'$. We will construct the exact versions of $p'_j$, $q'_j$ by a series of perturbations of $r_j$, $s_j$, each of which will be of size $O(\ep)$. We will keep track of explicit constants but will not optimize them.

Let us first apply Proposition \ref{prop_almost_projection} to each approximate projection $r_{2k-1}$, $r_{2k}$, $s_{2k}$, $s_{2k+1}$, with the Hilbert space being $\ran P_k$ or $\ran Q_k$, respectively. This yields a family of projections $r_j',s_j'$, with the first two quantities in equation \eqref{eq_quantBEEK_proof1}  exactly zero, and the last five quantities  not exceeding, say, $16\ep$. We also have
$$
\|r_j-r_j'\|\le 16\ep,\quad \|s_j-s_j'\|\le 16 \ep,\quad \forall j\in  \bZ/2N\bZ.
$$

We now construct the projections $p_j',q_j'$ by applying Corollary \ref{cor_two_projections} to each pair $r'_{2k-1},r'_{2k}$ and $s'_{2k},s'_{2k+1}$ to conjugate the even indexed projections into the complement of their odd companions inside the range of $P_k$ and $Q_k$, respectively. This produces unitary operators $U_k$, $V_k$ acting inside of $\ran P_k$, $\ran Q_k$ and satisfying
$$
U_k r'_{2k} U_k^{-1}\le \bone_{P_k}-r_{2k-1}',\quad V_k s'_{2k} V_k^{-1}\le \bone_{Q_k}-s_{2k+1}',\quad \|U_k-\bone_{P_k}\|\le 80\ep,\quad \|V_k-\bone_{Q_k}\|\le 80\ep.
$$
Define
$$
p_{2k-1}':=U_k r'_{2k-1} U_k^{-1},\quad p_{2k}':=P_k-p_{2k-1}';\quad q_{2k}':=V_k s'_{2k} V_k^{-1},\quad q_{2k+1}':=Q_k-q_{2k}'.
$$
We obtain equation \eqref{eq_quantgeek_property1} by combining the above estimates.

It remains to construct the unitary operator $W$. Using equation \eqref{eq_quantgeek_property1}, apply Proposition \ref{prop_two_projections} to the pair $p'_j,q'_j$, which will produce a unitary operator $W_j$ with
$$
\|W_j-\bone\|\le 400 \ep,\quad W_j p_j'W_j^{-1}=q_j'.
$$
The operator $W_j p_j'=q_j' W_j$, restricted to the range of $p_j'$, is a partial isometry between $\ran p_j'$ and $\ran q_j'$. Let
$$
W:=\sum_{j\in \bZ/2N\bZ}W_j p_j'=\sum_{j\in \bZ/2N\bZ}q_j' W_j.
$$
Clearly, $W$ is a unitary operator satisfying $Wp_j'=q_j'W$. In order to estimate $\|W-\bone\|$, note that for any vector $x$, we have
\begin{multline*}
\|Wx-x\|^2=\sum_j\|q_j'(Wx-x)\|^2=\sum_j\|W_j p'_j x-q_j' x\|^2\le \sum_j\left(\|(p_j'-q_j')x\|^2+\|(W_j-\bone)p_j'x\|\right)^2\\
\le 2\sum_j \|(p_j'-q_j')x\|^2+2\cdot 400^2 \ep^2\sum_j\|p_j'x\|^2\le \left(2\cdot 200^2 N\ep^2+2\cdot 400^2 \ep^2\right)\|x\|^2.\,\,\qedhere
\end{multline*}\end{proof}

\section{Isospectral invariant, homotopy, and the winding number} \label{s:isospec-homot-wind-numb}

In this section we discuss the isospectral invariant $\isospec(u,v)$ which was originally defined in \cite{BEEK}, state our first main result (the quantitative isospectral homotopy lemma), and use it to show that the isospectral invariant coincides with the winding number invariant.

\subsection{Definition of the isospectral invariant} \label{ss:isospec-def} \hfill\newline

Let $I, J\subset \bT$ be two intervals whose intersection comprises a single interval. We say that $I$ and $J$ are oriented counterclockwise provided $I\cap J$ prolongs counterclockwise from the (only) point of $\partial J$ which is contained in $I$. For example, if $\omega = e^{2\pi i/8},$ then $[\omega, \omega^3]$ and $[\omega^2, \omega^4]$ are oriented counterclockwise. 
A version of the following lemma, which serves here as the definition of the isospectral invariant, is established in \cite[Theorem 4.1]{BEEK}. However, we need the version with more explicit assumption in equation \eqref{eq_intervals_condition} and therefore include the proof.
\begin{prp}
\label{prop_isospec_definition}
There exist an absolute constant $C_1>0$ such that for any pair of unitary matrices $u,v\in {\mathrm U}(n)$
any two intervals $I,J\subset \bT$  oriented counterclockwise such that
\bee
\label{eq_intervals_condition}
%I\triangle J\neq \emptyset, \quad
1/10\ge \dist(\partial I,\partial J)\ge C_1\|[u,v]\|,
\ene
the operator $v\one_I(u)v^*\one_J(u)=\one_I(vuv^*)\one_J(u)$ is an approximate projection in the sense of Propostion $\ref{prop_almost_projection}$. The integer number
\bee
\label{eq_isospec_def}
\isospec(u,v):=\rank(\one_I(u)\one_J(u))-\rank_+(\one_I(vuv^*)\one_J(u))
\ene
is independent of the choice of $I,J$ within the assumptions of equation \eqref{eq_intervals_condition}, and remains constant if one continuously varies $u,v$ within the constraints on $\|[u,v]\|$ determined by equation \eqref{eq_intervals_condition}.
\end{prp}
\begin{proof}
Let $C_1 = 100C,$ where $C$ is the universal constant appearing in Lemma \ref{lemma_spectral2}. Equation \ref{eq_intervals_condition} then implies the estimate
\begin{equation*}
\frac{\|[u, v]\|}{\dist(\partial I, \partial J)}\leq \frac{1}{C_1}
\end{equation*}
Since $\|u-vuv^*\|=\|uv-vu\|$, Lemma \ref{lemma_spectral2} implies 
\bee
\label{eq_almost_commuting_projections}
\|[\one_I(vuv^*),\one_J(u)]\|\le \frac{C\|uv-vu\|}{\dist(\partial I,\partial J)}\leq \frac{C}{C_1}.
\ene
As a consequence, the real part $\Re \l(\one_I(vuv^*)\one_J(u)\r)$ is an approximate projection in the sense of Proposition \ref{prop_almost_projection} which is, say, $\frac{1}{100}$-close to a projection for $C_1$ large enough. On the other hand, equation \eqref{eq_almost_commuting_projections} also implies 
$$
\|\Re \l(\one_I(vuv^*)\one_J(u)\r)-\one_I(vuv^*)\one_J(u)\|\le \frac{1}{100}.
$$
As a consequence, $\one_I(vuv^*)\one_J(u)$ is within of $1/50$ norm distance from some orthogonal projection.

%; Here $C$ is the constant from Lemma \ref{lemma_spectral2} and $C_1$ is the constant that is up to our choice from \eqref{eq_intervals_condition}.

In order to establish independence from the choice of intervals, assume that $I=I_1\cup I_2$ is a union of two non-intersecting intervals such that the pair $I_1,J$ satisfies equation \eqref{eq_intervals_condition} and $I_2$ is completely contained inside of $J$ with the same boundary conditions. Then, one has
$$
\one_I(vuv^*)\one_J(u)=\one_{I_1}(vuv^*)\one_J(u)+\one_{I_2}(vuv^*)\one_J(u),
$$
where each term is $\frac{1}{50}$-close to a projection. As a consequence, $\rank_+$ is well-defined and additive, so that 
$$
\rank_+(\one_I(vuv^*)\one_J(u))=\rank_+(\one_{I_1}(vuv^*)\one_J(u))+\rank_+(\one_{I_2}(vuv^*)\one_J(u)).
$$
On the other hand, Lemma \ref{lemma_spectral1} (ii) implies that
$$
\|\one_{I_2}(vuv^*)\one_J(u)-\one_{I_2}(vuv^*)\|\le \frac{C\|u-vuv^*\|}{C_2\|uv-vu\|}=\frac{C}{C_1}<1.
$$
As a consequence of Proposition \ref{prop_two_projections},
$$
\rank_+(\one_{I_2}(vuv^*)\one_J(u))=\rank(\one_{I_2}(vuv^*))=\rank(\one_{I_2}(u))=\rank(\one_{I_2}(u)\one_J(u)),
$$
which implies that the additional terms in equation \eqref{eq_isospec_def} produced by $I_2$ cancel one another. The same arguments apply if $I_2$ is completely contained in $\bT\setminus J$, and for modifications of the interval $J$ of similar kind. It is easy to see (using the fact that $\dist(\partial I,\partial J)\leq 1/10$) that any pair of counterclockwise oriented intervals satisfying equation \eqref{eq_intervals_condition} can be transformed into any other pair by a series of, say, at most 10 such operators (see below). As a consequence, $\isospec(u,v)$ does not depend on the choice of intervals for this choice of $C_1$.

For the convenience of the reader, we provide more details about the above transformation. Within three moves one pair $I_1, J_1$ can be arranged so that $I_1\cup J_1$ is as short as possible; this choice can be made while fixing one endpoint of $I_1\cap J_1.$ Following the same principle, choose a preferred short arrangement for the target positively oriented pair $I_2, J_2$. Within four moves, the short arrangement we have can be transformed to the preferred short arrangement of the target pair. By reversing the moves which brings the preferred short arrangement to the target pair, we move the given pair to the target pair in ten moves. A more careful inspection of the argument may reduce the number of transformations, which will result in a small improvement in some constants.

It remains to establish homotopy invariance of the left hand side of equation \eqref{eq_isospec_def}. Suppose that $u(t),v(t)$ depend continuously on the parameter $t$. Since the spectrum of $u$ consists of finitely many eigenvalues that vary continuously with $t$, for every specific value of $t$ one can modify $I$, $J$ in a way that will not change the right hand side of the equation \eqref{eq_isospec_def}, but $\partial I\cup \partial J$ will contain no eigenvalues of $u$ (this modification depends on $t$, not necessarily in a continuous way). Therefore, all four projections involved in the definition of equation \eqref{eq_isospec_def} will be continuous at that value of $t$, which will imply continuity of the associated approximate projections and therefore their ranks. As a consequence, the left hand side of \eqref{eq_isospec_def} is continuous and therefore constant in $t$.
\end{proof}
As an example, we calculate the isospectral invariant of the {\it Voiculescu's unitaries} \cite{Voiculescu}. We note 
\begin{prp}
\label{prop_voiculescu_isospectral}
For $m=2,3,\ldots$, let $\omega:=e^{2\pi i/m}$ and
$$
\Omega_m := 
\begin{bmatrix}
\omega &  & & \\
 & \omega^2 & & \\
 & & \ddots & & \\
 & & & \omega^m
\end{bmatrix}\qquad 
S_k := 
\begin{bmatrix}
0 &1  & & \\
 & 0 &\ddots & \\
 & & \ddots &1 \\
 1& & & 0
\end{bmatrix}
$$
Then, the right hand side of \eqref{eq_isospec_def} is equal to $-1$. As a consequence, $\isospec(\Omega_m,S_m)=-1$ whenever it is well defined (say, for $m\ge 7$).
\end{prp}
\begin{proof}
A direct calculation shows (see, for example, \cite{Exel_Loring}) that for $m\geq 7$
$$
\|[S_m,\Omega_m]\|\le |1-e^{2\pi i/m}|\le \frac{10}{m}.
$$
As a consequence, the condition in equation \eqref{eq_intervals_condition} is satisfied for $k$ large enough. Note, however, that all four indicator functions, involved in equation \eqref{eq_isospec_def}, commute with each other, and therefore equation \eqref{eq_isospec_def} is well defined even without assuming equation \eqref{eq_intervals_condition}. Let $e_1,\ldots,e_m$ be the standard basis in $\C^m$, and let $p_j:=\langle e_j,\cdot \rangle e_j$ be the associated projections. If 
\bee
\label{eq_voiculescu_intervals}
I=[\omega^{s_1},\omega^{t_1}],\quad J=[\omega^{s_2},\omega^{t_2}],\quad 1\le s_1\le s_2\le t_1\le t_2\le m,
\ene
then one can easily check
$$
\one_I(\Omega_m)=p_{s_1}+\ldots+p_{t_1},\quad \one_J(\Omega_m)=p_{s_2}+\ldots+p_{t_2},\quad \one_I(S_m \Omega_k S_m^*)=p_{s_1-1}+\ldots+p_{t_1-1},
$$
where $p_0:=p_m$, and therefore the right hand side of equation \eqref{eq_isospec_def} is always equal to $-1$. We note that, assuming that $k$ is large enough, one can always find a pair of intervals among those in equation \eqref{eq_voiculescu_intervals} satisfying equation \eqref{eq_intervals_condition}. However, the reader can also check that the above calculation can be extended for any pair of intervals oriented counterclockwise, regardless of the norm of the commutator or distance estimates, assuming that the union of these intervals does not contain at least one diagonal entry of $\Omega_m$.
\end{proof}

\begin{rmk}
\label{rem_isospec_constraints}
Proposition \ref{prop_isospec_definition} serves as the definition of $\isospec(u,v)$. Since it relies on existence of at least one choice of intervals $I,J$ satisfying equation \eqref{eq_intervals_condition}, we have that $\isospec(u,v)$ is only defined under the assumption $\|[u,v]\|<\frac{1}{10 C_1}$. Rather than imposing this assumption each time $\isospec(u,v)$ is used, we will state the results that follow in a way that makes them trivial if $\|[u,v]\|$ is too large for $\isospec(u,v)$ to be defined.
\end{rmk}

\subsection{Proof of Lemma \ref{lemma_homotopy_intro}}\label{homotopy_lemma_proof}\hfill\newline

For two unitaries $u$ and $v$ with $\isospec(u,v)=0$, the aim of Lemma \ref{lemma_homotopy_intro} is to construct a piecewise smooth path $v_t\colon [0,1]\to {\mathrm U}(n)$ satisfying
$$
v_0=v,\quad v_1=\bone,\quad \|[v(t),u]\|\le C\|[u,v]\|^{2/5}.
$$
Let $u,v\in {\mathrm U}(n)$ and $\delta:=\|[u,v]\|$. Fix $2\le N\in \bN$, and let
$$
I_j:=\exp\left(\frac{i}{2N}[j,j+1)\right);\quad\lambda_j:=\exp\left(\frac{ij}{2N}\right),\quad j\in \mathbb Z
$$
be a system of intervals on the circle, with $\lambda_j$ being the endpoint of the corresponding interval. Normally, one would only consider the range $j=0,\ldots,2N-1$, but we include the possibility of arbitrary integer $j$ to take advantage of the periodicity of the exponential function. Let also
\bee
\label{eq_utilde_def}
q_j:=\one_{I_j}(u),\quad Q_k:=q_{2k}+q_{2k+1},\quad \tilde u:=\sum_{j=0}^{2N-1}\lambda_{j}q_j.
\ene
Clearly, we have
$$
\|\tilde u-u\|\le \frac{2\pi}{N}.
$$

Moreover, supposing that $z$ is a unitary operator commuting with all $Q_k$, $k=0,\ldots,N$, it follows that
\begin{equation}\label{eq-aux-one}
\|z\tilde u z^*-\tilde u\|\le \frac{8\pi}{N}.
\end{equation}
Indeed, since both $\tilde u$ and $z$ commute with each $Q_k$, it is sufficient to consider restrictions of both operators into the range of $Q_k$. However, such restriction of $\tilde u$ is a perturbation of a multiple of the identity of size at most $\frac{4\pi}{N}$. Any unitary conjugation by $z$ will be an operator of the same kind. Similarly to equation \eqref{eq_utilde_def}, define
$$
p_j:=\one_{I_j}(vuv^*)=vq_jv^*,\quad P_k:=p_{2k-1}+p_{2k}.
$$
The estimates in equation \eqref{eq-aux-one} also hold for a unitary defined from linear combinations of $p_k.$ 

Assume that $C\delta N<1/200$, where $C$ is the constant from Lemma \ref{lemma_spectral2} (which exceeds that from Lemma \ref{lemma_spectral1}). Then, one can easily check using Lemmas \ref{lemma_spectral1} and \ref{lemma_spectral2} that the families $P_k$, $Q_k$ satisfy the assumptions of Proposition \ref{quantBEEK}, 
\[\|[P_j,Q_k]\|\leq C\delta N\quad \|(P_j+P_{j+1} ) Q_j\|\leq C\delta N\quad \|(Q_{k-1}+Q_k)P_k\|\leq C\delta N.\] This produces another partition
$$
P_k=p_{2k-1}'+p_{2k}',\quad Q_k=q_{2k}'+q_{2k+1}',\quad k=0,\ldots,2N-1.
$$
Following the periodicity conventions, we have $p_{-1}=p_{2N-1}$. Moreover, each pair of intervals $I=I_{2k-1}\cup I_{2k}$, $J=I_{2k}\cup I_{2k+1}$ satisfies the assumptions of Proposition \ref{prop_isospec_definition}, and, appealing to \eqref{eq_better_estimate} in Proposition \ref{quantBEEK}, it follows that
\bee
\label{eq_isospec_consequence}
0=\isospec(u,v)=\rank p_j'-\rank p_j=\rank q_j'-\rank q_j,\quad \forall j=0,\ldots,2N-1,
\ene
assuming, say, $C_1\delta N<1/100$, where $C_1$ is the constant from Proposition \ref{prop_isospec_definition}. Since $p_j$ and $q_j$ are also conjugated by $v$, we have that all four projections $p_j,q_j,p_j',q_j'$ are of equal rank. Our next step shows that in fact there is also a proximity relation between $p_j$ and $p_j'$, and between $q_j$ and $q_j'$.

We will now introduce several unitary operators involved in the construction of the path.
\begin{enumerate}
	\item For $k=0,\ldots,N-1$, denote by $z_k$ any unitary operator acting inside of $\ran Q_k$ satisfying 
$$
q_{2k}'=z_k q_{2k} z_k,\quad z_k q_{2k+1}z_k^*=q_{2k+1}'.
$$
Note that the second property follows from the first one since $Q_k=q_{2k}+q_{2k+1}=q_{2k}'+q_{2k+1}'$. The existence of such $z_k$ follows from equation \eqref{eq_isospec_consequence}.
\item Let $z:=\oplus_k z_k$ be the combined unitary operator. There exists a self-adjoint operator $h$ with $\|h\|\le 2\pi$ such that $z=e^{ih}$. For $t\in [0,1]$, let $z_t:=e^{ith}$.
\item Let $w$ be the operator constructed during the earlier application of Proposition \ref{quantBEEK}, satisfying $\|w-\bone\|<100
\varepsilon \sqrt{N}$, where $\ep=CN\delta$. Denote by $w_t$ the shortest unitary path such that $w_0=\bone$, $w_1=w$.
\item Similarly to the first part, for $k=0,\ldots,N-1$, denote by $y_k$ any unitary operator acting inside of $\ran P_k$ satisfying 
$$
p_{2k}'=y_k p_{2k} y_k^*,\quad y_k p_{2k+1}y_k^*=p_{2k+1}'.
$$
\item Finally, let $y:=\oplus_k y_k=e^{ig}$, $g=g^*$ $\|g\|\le 2\pi$, and $y_t:=e^{itg}$.
\end{enumerate}

With the above preparations, consider the following path $\{\Gamma_t\colon t\in [0,3]\}$ of unitary operators,
\bee
\label{eq_path_definition}
\Gamma_t:=\begin{cases}
z_t,&t\in [0,1)\\
w_{t-1}z&t\in [1,2)\\
y_{t-2}wz&t\in [2,3].
\end{cases}
\ene
Equation \eqref{eq-aux-one} and the estimate $\|w-\bone\|<100
\varepsilon \sqrt{N}=100CN^{3/2}\delta$ imply
\bee
\label{eq_path_properties}
\Gamma_0\tilde u\Gamma_0^*=\tilde u,\quad \Gamma_3\tilde u\Gamma_3^*=\sum_{j=0}^{2N-1}\lambda_j p_j=v\tilde u v^*,\quad \|\Gamma_t\tilde u\Gamma_t^*-\tilde u\|\le C\left(N^{3/2}\delta+\frac{1}{N}\right),\quad t\in[0,3],
\ene
where $C$ is an absolute constant. %On the other hand, since $\Gamma_3$ transforms $q_k$ into $p_k$, we have
$$
%\Gamma_3\tilde u\Gamma_3^*=\sum_{j=0}^{2N-1}\lambda_j p_j=v\tilde uv^*.
$$
Let
$$
v_t:=\Gamma_t^*v,\quad t\in [0,3].
$$
Then
$$
[v_3,\tilde u]=0,\quad \|v_t \tilde u v_t^*-\tilde u\|=\|v\tilde u v^*-\Gamma_t\tilde{u}\Gamma_t^*\|\le \|\Gamma_t\tilde u\Gamma_t^*-\tilde u\|+\|\tilde u-v\tilde u v^*\|\le C\left(N^{3/2}\delta+\frac{1}{N}\right).
$$
We have connected $v_0=v$ with $v_3$ which commutes with $\tilde u$. It remains to connect $v_3$ to $\bone$ in a way that continues to commute with $\tilde u$. To do so, simply consider the restriction of $v_3$ into an eigenspace of $\tilde u$ and consider any unitary path inside that eigenspace that connects this restriction to the identity. For $t\in [3,4]$, denote by $v_t$ the resulting family of unitary operators, produced by performing that step in each eigenspace. The choice $N^{-1}=N^{3/2}\delta$, in order to balance the bounds, ultimately leads to
$$
\|[v_t,\tilde u]\|\le C\delta^{2/5},\quad\|[v_t,u]\|\le \|[v_t,\tilde u]\|+\|\tilde u-u\|\le (C+4\pi)\delta^{2/5}, \quad t\in [0,4].\,\,\qed
$$

\hfill

\subsection{Equivalence between the invariants} \label{ss:equiv-bet-inv}
\hfill\newline

We now establish the equivalence of the isospectral invariant with the winding number invariant. This equivalence appears in the literature \cite{Exel_Loring,BEEK} by identifying each obstruction with the Bott element. However, a more direct proof is possible following the same steps as \cite[Section 9]{BEEK}, without involving the Bott invariant, as the following quantitative corollary of Lemma \ref{lemma_homotopy_intro} demonstrates.
\begin{cor}
\label{cor_isospec_winding}
Let $u,v\in {\mathrm U}(n)$ with $C\|[u,v]\|^{2/5}<1$, where $C$ is the constant from Lemma $\ref{lemma_homotopy_intro}$. Define $w(u,v)$ to be the winding number of the curve
$$
\omega\colon [0,1]\to \bC\setminus\{0\},\quad \omega(t):= \det (t\cdot uv + (1-t)\cdot vu).
$$
Then $w(u,v)=\isospec(u,v)$.
\end{cor}

\begin{proof}
We note that that $w(\cdot,\cdot)$ is a homotopy invariant within the constraint $\|[u,v]\|<1$. As a consequence, under the assumptions of Lemma \ref{lemma_homotopy_intro} with $C\delta^{2/5}<1$, we have $w(u_t,v)=\mathrm{const}$ along the path provided by that lemma. As a consequence,
$$
w(u,v)=w(u_0,v)=w(u_1,v)=w(\bone,v)=0,\quad \text{whenever}\quad \isospec(u,v)=0.
$$
In order to treat the remaining cases, we will make use of Voiculescu's unitaries $S_m,\Omega_m\in U(m)$, introduced in Proposition \ref{prop_voiculescu_isospectral}; the latter, combined with \cite{Exel_Loring}, implies
$$
\|[S_m,\Omega_m]\|\le |1-e^{2\pi i/m}|\le \frac{10}{m};\quad w(S_m,\Omega_m)=\isospec(S_m,\Omega_m)=-1,\quad m\ge 7.
$$
Both $\isospec(\cdot,\cdot)$ and $w(\cdot,\cdot)$ are additive with respect to direct sums:
$$
\isospec(u_1\oplus u_2,v_1\oplus v_2)=\isospec(u_1,v_1)+\isospec(u_2,v_2);\quad w(u_1\oplus u_2,v_1\oplus v_2)=w(u_1,v_1)+w(u_2,v_2).
$$
Suppose now that $u,v$ satisfy the assumptions of the corollary and $\isospec(u,v)>0$. Let
$$
U:=u\oplus u',\quad V=v\oplus v',
$$
where $u',v'$ are made of $\isospec(u,v)$ copies of $\Omega_m,S_m$ with $m$ large enough so that one still has $C\|[U,V]\|^{2/5}<1$. Due to additivity, one has $\isospec(U,V)=0$, which implies, as in the beginning of the proof, $w(U,V)=0$, and therefore 
$$
w(u,v)=-w(u',v')=-\isospec(u',v')=\isospec(u,v).
$$
In order to deal with $\isospec(u,v)<0$, one can replace $S_m,V_m$ by $V_m^{-1}$, $S_m^{-1}$, respectively, which will reverse the signs of both invariants, and apply the same arguments.
\end{proof}

%\begin{rmk}
%\label{rem_iff}
%If $[u',v']=0$, then, clearly $w(u',v')=\isospec(u',v')=0$. As a consequence,
%$$
%w(u,v)=\isospec(u,v)=0\quad \text{whenever}\quad \|u-u'\|+\|v-v'\|\le \frac{1}{10}.
%$$
%We arrive to the well-known fact that the assumption %$w(u,v)=0$ is indeed necessary in Theorem \ref{th_main}.
%\end{rmk}

\subsection{Related results: a trace formula and infinite-dimensional amplifications}
\label{ss:related}
\hfill\newline

In \cite[Lemma 3.1]{Exel}, it is shown that the winding number invariant can be calculated using the following trace formula
\bee
\label{eq_trace}
w(u,v)=\frac{1}{2\pi i} \mathrm{Tr}\,\log (vuv^{-1}u^{-1}),
\ene
where $\log \colon \mathbb{C}\setminus(-\infty,0]\to \mathbb C$ denotes the principal branch and it is assumed that $\|uv-vu\|<2$. Note that, in the setting of Corollary \ref{cor_isospec_winding} one could also have started with \eqref{eq_trace} as the definition, since it evident that the right-hand-side is a homotopy invariance within the same constraints.

The following results is obtained in \cite[Theorem 2.5]{Lin_unitary}.
\begin{prp}
\label{prop_lin_unitary}
For every $\varepsilon>0$, there exists $\delta>0$ such that, if $u,v\in \mathrm{U}(n)$ satisfy $\|uv-vu\|<\delta$, then there exist unitary operators $U,V\in \mathrm{U}(\mathbb C^n\otimes \ell^2(\mathbb N))$ such that
$$
\|u\otimes \bone-U\|+\|v\otimes \bone-V\|<\varepsilon,\quad UV=VU.
$$
\end{prp}
Note that the result {\it does not} assume $w(u,v)=0$. In other words, the topological obstruction to obtaining commuting approximants disappears after an infinite amplification. The following lemma demonstrates a similar phenomenon that involves adding an infinite-dimensional identity operator.
\begin{lem}
\label{lemma_infinite_homotopy}
Let $u,v\in {\mathrm U}(n)$ with $C\|[u,v]\|^{2/5}<1$, where $C$ is the constant from Lemma $\ref{lemma_homotopy_intro}$. Then the pair $(u\oplus \bone_{\ell^2(\mathbb N)},v\oplus \bone_{\ell^2(\mathbb N)})$ is connected to the pair $(\bone_{\mathbb C^n\oplus \ell^2(\mathbb N)},\bone_{\mathbb C^n\oplus \ell^2(\mathbb N)})$ in $\mathrm{U}(\mathbb C^n\oplus \ell^2(\mathbb N))$ via a pair of paths $\{(u_t,v_t)\colon t\in [0,1]\}$ that satisfies 
\bee
\label{eq_homotopies_preserved}
C\|[u_t,v_t]\|^{2/5}<1,\quad t\in [0,1].
\ene
\end{lem}
\begin{proof}
Consider
$$
U:=(u\oplus u^*)\oplus (u\oplus u^*)\oplus\ldots=u\oplus (u^*\oplus u)\oplus (u^*\oplus u)\ldots \in \mathrm{U}(\mathbb C^n\oplus \ell^2(\mathbb N));
$$
$$
V:=(v\oplus v^*)\oplus (v\oplus v^*)\oplus\ldots=v\oplus (v^*\oplus v)\oplus (v^*\oplus v)\ldots\in \mathrm{U}(\mathbb C^n\oplus \ell^2(\mathbb N)).
$$
Clearly, $\|[U,V]\|=\|[u,v]\|$. Since $w(u,v)=-w(v^*,u^*)=-w(u^*,v^*)$, the winding number of the pair $(u\oplus u^*,v\oplus v^*)$ is zero, and Lemma \ref{lemma_homotopy_intro} provides a path connecting the pair $(u\oplus u^*,v\oplus v^*)$ with $(\bone_{{\mathbb C}^{2n}},v\oplus v^*)$. By using any path connecting $v\oplus v^*$ with $\bone_{{\mathbb C}^{2n}}$ in $\mathrm{U}(2n)$, one can see that the pair $(u\oplus u^*,v\oplus v^*)$ is connected to $(\bone_{{\mathbb C}^{2n}},\bone_{{\mathbb C}^{2n}})$ in $\mathrm{U}(4n)$, with both paths preserving the smallness of the commutators similar to \eqref{eq_homotopies_preserved}. Using the first representation for $U$ and $V$, one can transform the pair $(U,V)$ to $(\bone_{\mathbb C^n\oplus \ell^2(\mathbb N)},\bone_{\mathbb C^n\oplus \ell^2(\mathbb N)})$ by moving each copy of the  pair $(u\oplus u^*,v\oplus v^*)$ along the corresponding copy of the said path. Likewise, the second representation provides the homotopy between $(U,V)$ and the pair $(u\oplus \bone_{\ell^2(\mathbb N)},v\oplus \bone_{\ell^2(\mathbb N)})$. The proof can be concluded by combining all the above paths, with the estimate \eqref{eq_homotopies_preserved} following from the fact that it holds in each block separately.
\end{proof}

We note that the above construction involves a homotopy that affects both $U$ and $V$. While the original homotopy in Lemma \ref{lemma_homotopy_intro} only transforms $u$, the winding number is invariant under homotopies that affect $u$ and $v$, as long as the commutator remains small. Proposition \ref{prop_lin_unitary} and Lemma \ref{lemma_infinite_homotopy} demonstrate that, in these instances, the winding number obstruction ``disappears at infinity''.

\section{Commuting approximants and topological triviality} \label{s:commuting-approx}

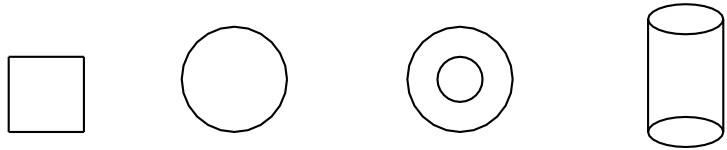
\begin{figure}[h]
\begin{tikzpicture}
% Box
       \draw[black, thick] (-3,1)--(-2,1);
       \draw[black, thick](-2,0)--(-3,0);
       \draw[black, thick] (-3, 0)--(-3,1);
       \draw[black, thick] (-2,1)--(-2, 0);

% Disk
        \draw[black, thick, domain=-0:360] plot ({0.7*cos(\x)},{0.7+0.7*sin(\x)});

%Annulus
        \draw[black, thick, domain=-0:360] plot ({3+0.3*cos(\x)},{0.7+0.3*sin(\x)});
        \draw[black, thick, domain=-0:360] plot ({3+0.7*cos(\x)},{0.7+0.7*sin(\x)});

        % \draw[gray, thick] (3.3,0.7)--(3.7,0.7);
        % \draw[gray, thick] (2.7,0.7)--(2.3,0.7);
    
%Cylinder
      \draw [black,thick,domain=-0:360] plot ({6+0.5*cos(\x)}, {1.5+0.2*sin(\x)});
       \draw [black,thick,domain=-0:360] plot ({6+0.5*cos(\x)}, {0.2*sin(\x)});
       \draw[black, thick] (5.5, 0)--(5.5, 1.5);
       \draw[black, thick] (6.5, 0)--(6.5, 1.5);

\end{tikzpicture} 
\caption{Noncommutative ``Change of Variables,'' see Remark \ref{rmk:joint_spectra}. }
\end{figure}
In this section we prove Theorem \ref{th_gap_opening_intro}. By applying the quantitative version of Lin's theorem we first reduce the problem to the question of opening a uniform gap in the spectrum of one of the unitaries. However, we do not apply this reduction as-is. In order to obtain sharper bounds we apply this reduction to sufficiently large ampliated pairs of matrices, and a proof is furnished by keeping track of quantitative estimates when conducting dimension-reductions of the ampliated pairs along the process.

\subsection{Almost commuting matrices with spectral gaps} \label{ss:prep} \hfill\newline

The following proposition combines improved versions of several well-known results,  see, for example \cite[Section 2]{HL}. The original proof in \cite{HL} relied on the outcome of \cite{H_orig} as a ``black box''. In the following proof, which is included mainly for the convenience of the reader, we replace the use of \cite{H_orig} by the stronger result of \cite{KS}.
\begin{prp}
\label{prop_top_trivial}
Suppose that $t, s\in M_n(\mathbb C)$ satisfy one of the following properties:
\begin{enumerate}
\item $t=t^*$ and $s=s^*$;
\item $s=t^*$;
\item $s=t^*$ and $\bone\le |t|\le 3\cdot\bone$;
\item $t=t^*$, $\|t\|\le 1$, and $s\in {\mathrm U}(n)$.
\end{enumerate}
Then there exist matrices $s', t'$ satisfying the same respective conditions such that 
$$
[t', s']=0\quad \|t-t'\|+\|s-s'\|\leq C\big\|[t, s]\big\|^{1/2},\quad \|t'\|\le \|t\|,\quad \|s'\|\le \|s\|.
$$
\end{prp}
\begin{proof}
The first claim is the original form of Lin's theorem, in the quantitative version from \cite{KS}. As explained earlier in the Introduction, the second claim follows from the first one applied to the real and imaginary parts of $s$. Note that in the first two claims we do not assume $s$ and $t$ to be contractions, since the exponent $1/2$ makes the problem scale-invariant. The matrices $s'$ and $t'$ can be normalized in order to satisfy the norm requirement, which may result in a modification of $C$.

In order to obtain the third claim, start from applying the second claim and assume (without loss of generality, perhaps after a modification of $C$) that $\|t-t'\|=\|s-s'\|\le 1/10$. Since the function $z\mapsto|z|$ is operator Lipschitz on any compact subset of $\mathbb C\setminus\{0\}$, we have that
$$
\||s|-|s'|\|=\||t|-|t'|\|\le C'\|[t,s]\|^{1/2}\le 1/10,
$$
where, again, the last inequality can also be assumed without loss of generality by choosing $C$ large enough. As a consequence,
$$
\left(1-C'\|[t,s]\|^{1/2}\right)\bone\le |t'|\le 3\cdot\bone.
$$
In order to satisfy the second part of the third claim, let
$$
g(z):=\begin{cases}
\frac{z}{|z|},&|z|\leq 1,\\
z,& 1<|z|<3,\\
3\frac{z}{|z|} & |z|\geq 3
\end{cases}
$$
defined on $\mathbb C\setminus \{0\}$. Since both $s'$ and $t'$ are normal, the standard functional calculus (spectral mapping theorem) implies $1\cdot\bone\le |g(t')|\le 3\cdot\bone$, as well as
$$
\|s-g(s')\|+\|t-g(t')\|\le \|s-s'\|+\|s'-g(s')\|+\|t-t'\|+\|t'-g(t')\|\le C''\|[t,s]\|^{1/2}.
$$

In order to establish the fourth claim, let $a:=s(t+2\cdot\bone)$. Since $\bone\le t+2\cdot\bone\le 3\cdot\bone$, we have that $a$ and $a^*$ satisfy the assumptions of the third claim with, say,
$$
\|[a,a^*]\|\le 10\|[s,t]\|^{1/2}.
$$
Let $b$ be a normal operator obtained by applying the third claim, and
$$
t':=|b|-2=(b^*b)^{1/2}-2,\quad s':=b|b|^{-1}=b(b^*b)^{-1/2}.
$$
We have
$$
t-t'=(b^*b)^{1/2}-(a^*a)^{1/2},\quad s-s'=(b-a)(b^*b)^{-1/2}+a\l((b^*b)^{-1/2}-(a^*a)^{-1/2}\r).
$$
Now, the claim follows from the estimates
$$
\|(b^*b)^{1/2}-(a^*a)^{1/2}\|+\|(b^*b)^{-1/2}-(a^*a)^{-1/2}\|\le C\|b^*b-a^*a\|\le 20C\|b-a\|\le C'\|[s,t]\|,
$$
in which we used the facts that 
$$
\bone \le b^*b\le 10\cdot \bone,\quad \bone \le a^*a\le 10\cdot \bone,
$$
and that both functions $x\mapsto x^{1/2}$ and $x\mapsto x^{-1/2}$ are operator Lipschitz on $[1,10]\subset \mathbb R$.
\end{proof}
\begin{rmk}\label{rmk:joint_spectra}
It can be helpful to associate each of the cases considered in Proposition \ref{prop_top_trivial} to a geometric picture of the joint spectra of $s$ and $t$ in the case where their commutator would be exactly zero.
\begin{enumerate}
\item In the case $t=t^*$ and $s=s^*$, assume without loss of generality that $\|s\|\le 1$ and $\|t\|\le 1$. Then, the joint spectrum is contained in the square $[-1,1]\times [-1,1]$.
\item In the case $s=t^*$ assume (again, without loss of generality) that $\|s\|=\|t\|\le 1$. Then, the joint spectrum can be identified with the spectrum of $s$ and is contained in the unit disk.
\item In the case $s=t^*$ and $\bone\le |t|\le 3\cdot\bone$, the joint spectrum of is contained in an annulus which is a subset of the (rescaled) disk from the previous case.
\item In the case $t=t^*$, $\|t\|\le 1$, and $s\in \mathrm{U}(n)$, the joint spectrum can be considered as a subset of the cylinder, which is naturally homeomorphic to the annulus from the previous case.
\end{enumerate}
\end{rmk}
It is not hard to show that if one replaces $t$ in the last part of Proposition \ref{prop_top_trivial} by a unitary operator with a spectral gap, the approach would still work. However, the size of the gap will appear in the final estimate. The following lemma quantifies this observation. We note that while it is natural to expect the extra factor $\varepsilon^{-1/2}$, where $\varepsilon$ is the size of the gap, it requires a somewhat delicate estimate of the operator Lipschitz norm of a branch of the argument function, which does not offer a natural rescaling property. This estimate is provided by Lemma \ref{lemma_arg_lipschitz} in Appendix \ref{appendix-A}. Modulo that estimate, the proof is very short, as illustrated by Figure 2, where the cylinder represents Case (4) in Proposition \ref{prop_top_trivial}.

\begin{lem}
\label{lemma_almost_commuting_gapped}
Suppose, $u$ and $v$ are two unitary matrices. Assume, in addition, that $u$ has a spectral gap of size $\rho>0$. Then there exist unitary $u',v'$ such that
\bee
\label{eq_cor_gapped_cylinder}
[u',v']=0, \quad \|u-u'\|+\|v-v'\|\le C\rho^{-1/2}\|[u,v]\|^{1/2}.
\ene
\end{lem}
\begin{proof}
Without loss of generality, one can assume that the spectral gap of $u$ is around $-1$.	Let $x:=\frac{1}{2\pi}\arg u=\frac{1}{2\pi}\arg_{\rho}u$ as considered in Appendix \ref{appendix-A}. From Lemma \ref{lemma_arg_lipschitz} and (OL6), we have
$$
\|[x,v]\|\le C \rho^{-1}\|[u,v]\|.
$$
Apply the last claim of Proposition \ref{prop_top_trivial} with $t=x$ and $s=v$, thus arriving to a commuting pair of $x'$ and $v'$. Since the function $x\mapsto e^{i x}$ is operator Lipschitz on $\mathbb R$, we have that $u'=e^{2\pi ix'}$ and $v'$ satisfy equation \eqref{eq_cor_gapped_cylinder}.
\end{proof}
\begin{rmk}
\label{remark_osborne}
In \cite{Osborne}, an analogue of Lemma \ref{lemma_almost_commuting_gapped} was proved for the case where both matrices have spectral gaps. It appears that a modification of the arguments in \cite{Osborne} along the same lines as in Proposition \ref{prop_top_trivial} ultimately leads to the same conclusion as in Lemma \ref{lemma_almost_commuting_gapped} (in other words, it is not necessary for the second matrix to have a gap). In particular, \cite[Section 2]{Osborne} provides a different construction of a branch of the logarithm with the same properties as the one in Lemma \ref{lemma_arg_lipschitz}.
\end{rmk}

A potential strategy for proving Theorem \ref{th_main} would now be to reduce it to Lemma \ref{lemma_almost_commuting_gapped}, which suggest the following question: given two unitary matrices $u,v$ with small commutator, is it possible to ``open a gap'' in the spectrum of $u$, while preserving smallness of the commutator with $v$ in an appropriate (dimension-independent) sense? Clearly, a meaningful positive answer to this question would provide a path to proving Theorem \ref{th_main}, and therefore rely on vanishing of the winding number $w(u,v)$. The following result, established (with a short proof) in \cite{Phillips}, shows that such gap opening is possible, regardless of the value of $w(u,v)$, if one is allowed to consider an amplification by $u^*$.
% \begin{figure}
% \begin{tikzpicture}
% % %% Box
% %         \draw[red, thick] (0,1)--(1,1);
% %         \draw[red, thick](1,0)--(0,0);
% %         \draw[blue, thick] (0, 0)--(0,1);
% %         \draw[blue, thick] (1,1)--(1, 0);
% %%Cylinder
%        \draw [red,thick,domain=-0:360] plot ({3+0.5*cos(\x)}, {1.5+0.2*sin(\x)});
%         \draw [red,thick,domain=-0:360] plot ({3+0.5*cos(\x)}, {0.2*sin(\x)});
%         \draw[gray, thick] (2.5, 0)--(2.5, 1.5);
%         \draw[blue, thick] (3.5, 0)--(3.5, 1.5);
% %% Torus
%         \draw [black,thick,domain=-0:360] plot ({6+0.7*cos(\x)}, {0.5+0.5*sin(\x)});
%         \draw [black,thick,domain=210:330] plot ({6+0.5*cos(\x)}, {0.85+0.5*sin(\x)});
%         \draw [black,thick,domain=45:135] plot ({6+0.5*cos(\x)}, {0.15+0.5*sin(\x)});
%         \draw [red, thick]({6+0.5*cos(265)}, {0.85+0.5*sin(265)}) arc (100:260:0.18) ({6+0.5*cos(265)}, {0.50+0.5*sin(265)}) ;

% \end{tikzpicture} 
% \caption{noncommutative ``Change of variables'' from cylinders to tori}
% \end{figure}

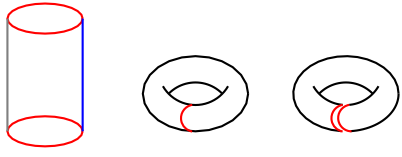
\begin{figure}
\begin{tikzpicture}

%%Cylinder
       \draw [red,thick,domain=-0:360] plot ({4+0.5*cos(\x)}, {1.5+0.2*sin(\x)});
        \draw [red,thick,domain=-0:360] plot ({4+0.5*cos(\x)}, {0.2*sin(\x)});
        \draw[gray, thick] (3.5, 0)--(3.5, 1.5);
        \draw[blue, thick] (4.5, 0)--(4.5, 1.5);

%% Torus
        \draw [black,thick,domain=-0:360] plot ({6+0.7*cos(\x)}, {0.5+0.5*sin(\x)});
        \draw [black,thick,domain=210:330] plot ({6+0.5*cos(\x)}, {0.85+0.5*sin(\x)});
        \draw [black,thick,domain=45:135] plot ({6+0.5*cos(\x)}, {0.15+0.5*sin(\x)});
        \draw [red, thick]({6+0.5*cos(265)}, {0.85+0.5*sin(265)}) arc (100:260:0.18) ({6+0.5*cos(265)}, {0.50+0.5*sin(265)}) ;

%%Torus:Outer Arc
        \draw [black,thick,domain=0:265] plot ({6+2+0.7*cos(\x)}, {0.5+0.5*sin(\x)});
        \draw [black,thick,domain=275:360] plot ({6+2+0.7*cos(\x)}, {0.5+0.5*sin(\x)});
%% Torus: Inner Arc
      \draw [black,thick,domain=210:265] plot ({6+2+0.5*cos(\x)}, {0.85+0.5*sin(\x)});
        \draw [black,thick,domain=275:330] plot ({6+2+0.5*cos(\x)}, {0.85+0.5*sin(\x)});
%%Torus: Upper Arc
        \draw [black,thick,domain=45:135] plot ({6+2+0.5*cos(\x)}, {0.15+0.5*sin(\x)});
%% Nodes
        \node at ({6+2+0.5*cos(265)}, {0.85+0.5*sin(265)}) (one){};
        \node at ({6+2+0.5*cos(265)}, {0.50+0.5*sin(265)}) (two){};
        \node at ({6+2+0.5*cos(275)}, {0.85+0.5*sin(275)}) (three){};
        \node at ({6+2+0.5*cos(275)}, {0.50+0.5*sin(275)}) (four){};
%% Cuts
        \draw [red, thick]({6+2+0.5*cos(265)}, {0.85+0.5*sin(265)}) arc (100:260:0.18) ({6+0.5*cos(265)}, {0.50+0.5*sin(265)}) ;
        \draw [red, thick]({6+2+0.5*cos(275)}, {0.85+0.5*sin(275)}) arc (100:270:0.1775) ({6+0.5*cos(275)}, {0.50+0.5*sin(275)}) ;
        % \draw [red,thick,domain=90:270] plot ({6+0.01*cos(\x)}, {0.1+0.125*sin(\x)});

\end{tikzpicture} 
\caption{When the obstruction vanishes, we decouple red edges.}
\end{figure}

\begin{prp}
\label{prop_unitary_double}
Let $\mathcal A$ be a unital $C^{*}$-algebra, and $u\in U(\mathcal A)$ be a unitary element. For $\ep>0$, let
\bee
\label{eq_epsilon_rotation}
w(\ep):=\begin{pmatrix}
	u&0\\0&I\end{pmatrix}
	\begin{pmatrix}
	\cos(\pi/2-\ep)&\sin(\pi/2-\ep)\\-\sin(\pi/2-\ep)&\cos(\pi/2-\ep)\end{pmatrix}
	\begin{pmatrix}
	u^*&0\\0&I\end{pmatrix}
	\begin{pmatrix}
	\cos(\pi/2-\ep)&-\sin(\pi/2-\ep)\\\sin(\pi/2-\ep)&\cos(\pi/2-\ep)\end{pmatrix}.
\ene
Then, for $0<\ep<1/10$, we have:
$$
\left\|w(\ep)-\begin{pmatrix}
	u&0\\0&u^*\end{pmatrix}\right\|<3\ep,\quad \|(w(\varepsilon)+\bone)^{-1}\|\le \ep^{-1}.
$$
\end{prp}
\begin{rmk}
\label{rem_almost_commuting_amplification}
From equation \eqref{eq_epsilon_rotation}, one can easily observe
$$
\|\l[w(\varepsilon),\mathrm{diag}(v,v)\r]\|\le 2 \|[u,v]\|,\quad \forall v\in \mathcal{A}.
$$
\end{rmk}
Note that $w(0)=\text{diag}(u, u^*)$. Therefore, equation \eqref{eq_epsilon_rotation} provides a perturbation of $\text{diag}(u, u^*)$ of size $O(\ep)$ that opens a gap of size $\ep$ (the extra factors are mostly to absorb the difference between arc length and diameter). As mentioned earlier, there are no additional assumptions on $\mathcal A$ and $u$. In principle, any topological obstruction (such as the winding number invariant), preventing the opening of a gap in the spectrum of $u$, exactly cancels with that of $u^*$.
\begin{rmk}
\label{remark_macro}
While the above-mentioned gap opening strategy would work and produce a meaningful version of Theorem \ref{th_main}, the actual construction is somewhat more delicate. Lemma \ref{lemma_almost_commuting_gapped} is first applied to the amplified matrices with  $\rho=\delta^{1/3}$ in the beginning of Subsection \ref{ss:amplified-comm}. In this case, the pair of unitary operators has a ``microscopic gap'', that is, the one whose size goes to $0$ as $\delta\to 0$. However, the result of this application allows one to partition the spectrum of one of the matrices in a such a way that further applications involve only ``macroscopic gaps'' of size that is independent of $\delta$. The case of a macroscopic gap is closer to the kind of gap considered in Proposition \ref{prop_top_trivial}, since it does not require precise control of estimates as the gap size goes to zero, and the commutator losses are the same as in Proposition \ref{prop_top_trivial}, since the factor $\rho^{-1}$ is bounded by a constant.
\end{rmk}
\subsection{Proof of Theorem \ref{th_gap_opening_intro}: the amplified almost commuting pair with a spectral gap} \label{ss:amplified-pair} \hfill\newline 

As a reminder, we assume that $u,v\in \mathrm{U}(n)$ are unitary matrices such that $\|[u,v]\|\le \delta$ and that there exists a continuous path $\{u_t\colon t\in [0,1]\}$ satisfying
$$
u_0=u,\quad  u_1=\bone,\quad \|[u_t,v]\|\le \delta \quad \forall t\in [0,1].
$$
To prove Theorem \ref{th_gap_opening_intro} we must construct unitary matrices $u'$, $v'$ such that
$$
[u',v']=0,\quad \|u-u'\|+\|v-v'\|\le C\delta^{1/3}.
$$   
Note that, in view of Lemma \ref{lemma_almost_commuting_gapped}, it would be sufficient to transform $u$ and $v$ into a pair of almost commuting unitary matrices where, in addition, one of them has a spectral gap. Following the strategy of retracing the steps of \cite{Lin_exprank}, we will first consider amplifications of the original pair $(u,v)$ in some space of larger dimension, in which creating a gap would be easier. Afterwards, we carefully descend back into the original space in order to produce a commuting pair. The descent will be performed in two steps, with each step being referred to as ``dimension reduction''.

For the sake of brevity, we will use the notation $\mathrm{A}:=\mathrm{M}_n(\C)$, so that $u,v\in \mathrm{A}$.  For $k\in \mathbb{N}$, we also denote by $\mathrm{A}_k = \mathrm{M}_k(\mathrm{A}) = \mathrm{A}\otimes \mathrm{M}_k(\mathbb{C})$ the $k$th matrix amplification of $\mathrm{A}.$ 

Now, fix $\varepsilon>0.$ Under the assumptions of Theorem \ref{th_gap_opening_intro}, choose a (large) integer $d\in \mathbb N$ and subdivide the path $u_t$ into $d$ segments of length at most $\ep$; that is, construct a sequence of matrices $u_0,\ldots,u_d$ such that
$$
u_0=u,\quad u_d=\bone,\quad \|[u_j,v]\|\le \delta,\quad \|u_{j+1}-u_{j}\|\le \ep\,\,\,\text{for}\,\,\,j=0,\ldots d-1.
$$
Along the lines of \cite{Lin_exprank}, let
$$
u_{\mathrm{path}}:=\mathrm{diag}\{u_1^*,u_1,u_2^*,u_2,\ldots,u_{d-1}^*,u_{d-1},\bone_{\mathrm A}\}\in \mathrm{A}_{2d-1};
$$
$$
u_{\mathrm{amp}}:=\mathrm{diag}\{u,u^*,u_1,u_1^*,u_2,u_2^*,\ldots,u_{d-1},u_{d-1}^*\}\in \mathrm{A}_{2d}.
$$
From the construction, it follows
$$
\|u\oplus u_{\mathrm{path}}-u_{\mathrm{amp}}\|\le \ep;\quad \|[u_{\mathrm{path}},v_{\mathrm{path}}]\|\le \delta, \quad \|[u_{\mathrm{amp}},v_{\mathrm{amp}}]\|\le \delta,
$$
where 
$$
v_{\mathrm{path}}:=v\otimes \bone_{(2d-1)\times (2d-1)}\in  \mathrm{A}_{2d-1},\quad v_{\mathrm{amp}}:=v\otimes \bone_{2d\times 2d}=v\oplus v_{\mathrm{path}}\in  \mathrm{A}_{2d}
$$ are the amplifications of $v$ of the corresponding dimensions.

If one ignores the identity matrix in $u_{\mathrm{path}}$, both $u_{\mathrm{amp}}$ and $u_{\mathrm{path}}$ are made out of blocks of the form that allows the application of Proposition \ref{prop_unitary_double}. For $0<\ep<1/10$, denote the results of applying that proposition inside each block by $u_{\mathrm{path}}(\ep)$ and $u_{\mathrm{amp}}(\ep)$, respectively, without altering the last block entry $\bone$ in $u_{\mathrm{path}}$. It is easy to see that they satisfy
$$
\|u_{\mathrm{path}}(\ep)-u_{\mathrm{path}}\|\le 3\ep,\quad \|u_{\mathrm{amp}}(\ep)-u_{\mathrm{amp}}\|\le 3\ep; \quad \|[u_{\mathrm{path}}(\ep),v_{\mathrm{path}}]\|\le 2\delta, \quad \|[u_{\mathrm{amp}}(\ep),v_{\mathrm{amp}}]\|\le 2\delta.
$$
The last two estimates follow from the precise forms of $u_{\mathrm{path}}(\varepsilon)$ and $u_{\mathrm{amp}}(\varepsilon)$ in Proposition \ref{prop_unitary_double} on each $2\times 2$ block. Moreover, both $u_{\mathrm{path}}(\ep)$ and $u_{\mathrm{amp}}(\ep)$ have spectral gaps of size $\ep$ near $-1$.

It also follows that 
\[\|u\oplus u_{\mathrm{path}}(\ep)-u_{\mathrm{amp}}(\ep)\|\leq \|u\oplus u_{\mathrm{path}}-u_{\mathrm{amp}}\|+6\varepsilon\le 7\varepsilon.\]
Recall the earlier notation: if $q$ is a projection and $a$ is an operator, then:
$$
qa|_q=(qaq)|_{\ran q}.
$$
If $p$ is the projection onto the copy of $\mathrm{A}_d$ associated to the top left corner, we have, in the above notation,
$$
\|u-pu_{\mathrm{amp}}(\varepsilon)|_p\|=\|p(u_{\mathrm{amp}}-u_{\mathrm{amp}}(\varepsilon))p\|\le 3\varepsilon. 
$$
%One considers $u''(\varepsilon)$ as an amplified approximant to $u,$ in the sense that the $(1,1)$ corner of $u''(\varepsilon)$ is close to our original unitary $u.$  
As a consequence, one can consider $u_{\mathrm{amp}}(\varepsilon)$ as an amplified approximant to $u$, which has a spectral gap of size approximately $\varepsilon$, but also retains approximate commutation relation with $v_{\mathrm{amp}}$ of magnitude approximately $\delta$, where $\varepsilon$ and $\delta$ are not bound by any additional relation.

\subsection{Proof of Theorem \ref{th_gap_opening_intro}: the amplified commuting pair.} \label{ss:amplified-comm} \hfill\newline 

%In the same spirit as the previous paragraph, we next seek an amplified approximant to $v.$ 
With the assurance of a unitary $u_{\mathrm{amp}}(\varepsilon)$ admitting a gap in its spectrum, we apply Lemma \ref{lemma_almost_commuting_gapped} with $\rho=\varepsilon$ to $u_{\mathrm{amp}}(\varepsilon)$ and $v_{\mathrm{amp}}$ to obtain the commuting pair
\[[u'_{\mathrm{amp}}(\varepsilon), v'_{\mathrm{amp}}(\varepsilon)]=0\]
subject to 
\[\|u'_{\mathrm{amp}}(\varepsilon)-u_{\mathrm{amp}}(\varepsilon)\|+\|v'_{\mathrm{amp}}(\varepsilon)-v_{\mathrm{amp}}\|\leq C\varepsilon^{-1/2}\delta^{1/2}.\]
Here, $v_{\mathrm{amp}}'(\varepsilon)$ serves as an amplified approximant to $v$, which carries along a commuting unitary $u'_{\mathrm{amp}}(\varepsilon).$ 
We likewise generate the commuting pair $(u'_{\mathrm{path}}(\varepsilon), v'_{\mathrm{path}}(\varepsilon))$ with identical estimates, from the pair $(u_{\mathrm{path}}(\varepsilon), v_{\mathrm{path}})$. On account of the essential feature of commutation for the unitaries and the estimate 
\bee
\label{eq_expense_estimate}
\|u\oplus u'_{\mathrm{path}}(\varepsilon)-u'_{\mathrm{amp}}(\varepsilon)\|\leq 7\varepsilon + C\varepsilon^{-1/2}\delta^{1/2},
\ene we can now abandon our original amplified approximant $u_{\mathrm{amp}}(\varepsilon)$ of $u$ in favor of $u'_{\mathrm{amp}}(\varepsilon)$, at the expense of the estimate \eqref{eq_expense_estimate}. 

Since $[u'_{\mathrm{amp}}(\varepsilon),v'_{\mathrm{amp}}(\varepsilon)]=0$, we have $[f(u'_{\mathrm{amp}}(\varepsilon)),v'_{\mathrm{amp}}(\varepsilon)]=0$ for every Borel function $f\colon \C\to [0,1]$. As a consequence, 
\bee
\label{eq_f_no_loss}
\l\|[f(u'_{\mathrm{amp}}(\varepsilon)),v_{\mathrm{amp}}]\r\|\le 2\l\|v'_{\mathrm{amp}}(\ep)-v_{\mathrm{amp}}\r\|\le C\varepsilon^{-1/2}\delta^{1/2}.
\ene
In particular, this allows us to create additional gaps in the spectrum of $u'_{\mathrm{amp}}(\varepsilon)$ or split into parts using spectral projections, without further losses in the commutator norm. This idea will be used several times in the next subsections.

\subsection{Proof of Theorem \ref{th_gap_opening_intro}: the first dimension reduction.}\label{ss:firstreduction}\hfill\newline 

\begin{figure}[h]
\begin{tikzpicture}
%% Higher Circle
\draw[black, thick, domain=0:360] plot ({-0.5+0.5*cos(\x)}, {0.5+0.5*sin(\x)});
      \filldraw[black] ({cos(180)}, {0.5+sin(180)}) circle (0pt) node[anchor=east]{{\tiny $p$ }};
%% Lower circle
\draw[red, thick, domain=-88:88] plot ({1+cos(\x)}, {-1+sin(\x)});
      \filldraw[black] (2, -1) circle (0pt) node[anchor=west]{{\tiny $r$ }};
\draw[blue, thick, domain=92:268] plot ({1+cos(\x)}, {-1+sin(\x)});
      \filldraw[black] (0, -1) circle (0pt) node[anchor=east]{{\tiny $q$ }};
\end{tikzpicture} 
\caption{A ``high order'' spectral decomposition for $u\oplus u'_{\mathrm{path}}(\varepsilon)$ into the unitary $u
\in A$ and two semicircular arcs dividing $u'_{\mathrm{path}}(\varepsilon)$. After slight modifications (see equations \eqref{eq_sigma_properties} and \eqref{eq_tau_close_1} below), the spectral projections for the black, red, and blue arcs implement dimension reductions to $u'_{\mathrm{amp}}(\ep)$ (with a controlled penalty to commutators). }
\end{figure}
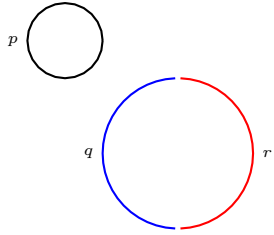
We now briefly summarize our position. Given a pair of unitaries $u, v\in \mathrm{A}$, a path $u_t$ from $u$ to the identity with $\|[u_t, v]\|<\delta$, and $\varepsilon<1/10$, we have found an integer $d$, unitaries $u'_{\mathrm{amp}}(\varepsilon), v'_{\mathrm{amp}}(\varepsilon)\in \mathrm{A}_{2d}$, and unitaries $u'_{\mathrm{path}}(\varepsilon), v'_{\mathrm{path}}(\varepsilon)\in \mathrm{A}_{2d-1}$ which obey the relations 

\bee
\label{eq_summary_3}
\|u\oplus u'_{\mathrm{path}}(\ep)-u'_{\mathrm{amp}}(\ep)\|\le \ep+6\ep+C\ep^{-1/2}\delta^{1/2},\ \|[u'_{\mathrm{path}}(\ep),v'_{\mathrm{path}}(\ep)]\|=\|[u'_{\mathrm{amp}}(\ep),v'_{\mathrm{amp}}(\ep)]\|=0,
\ene
\bee
\label{eq_summary_4}
\|v'_{\mathrm{path}}(\ep)-v_{\mathrm{path}}\|+\|v'_{\mathrm{amp}}(\ep)-v_{\mathrm{amp}}\|\le C \ep^{-1/2}\delta^{1/2}.
\ene
Up to amplification to $\mathrm{A}_{2d},$ we have found commuting approximants. The main challenge now is to descend into the original space. Clearly, if we simply apply the compression by the projection associated to the top left corner, the operators will be almost unitary and almost commuting, but we lose control over the spectral gap, in some sense returning to the original problem. A similar difficulty appears in \cite{Lin_exprank}: compression of an element with finite spectrum does not have to be close to an element with finite spectrum. However, compression of a self-adjoint element is still self-adjoint, so it will always be of this kind. The same holds for unitary elements with large gaps. So, one can prepare for the compression by creating such a gap.

We will use spectral projections of $u\oplus u'_{\mathrm{path}}(\varepsilon)$ to reduce the extra dimensions of $(u'_{\mathrm{amp}}(\varepsilon), v'_{\mathrm{amp}}(\varepsilon))$ with the aim of drawing the compressions back into commutation through repeated application of Lemma \ref{lemma_almost_commuting_gapped}. To ease the burden of notation, we will proceed denoting the unitaries above by $u'_{\mathrm{amp}}, v'_{\mathrm{amp}},$ and $ u'_{\mathrm{path}}, v'_{\mathrm{path}}$, dropping the dependence on $\varepsilon$.

Concerning the operator $u\oplus u'_{\mathrm{path}}$, write $p$ for the projection onto the range of $u$ in $\mathrm{A}_{2d}$ (that is, the top left corner). Next, we split the spectrum of $u'_{\mathrm{path}}$ into two semi-circles: let 
\[
q:={1}_{\{\Re z\le 0\}}(u'_{\mathrm{path}}),\quad r:={1}_{\{\Re z>0\}}(u'_{\mathrm{path}});\quad u'_-:=qu'_{\mathrm{path}}|_q,\quad u'_+:=ru'_{\mathrm{path}}|_r
\]
denote the corresponding spectral projections and components of $u'_{\mathrm{path}}$ in the corresponding subspaces, so that
\[u\oplus u'_{\mathrm{path}} = u\oplus u'_-\oplus u'_+,\quad \bone:=\bone_{A_{2d}}=p+q+r. 
\]

The goal of the first dimension reduction is to eliminate the blocks associated to the projection $r$, using a certain spectral projection of $u'_{\mathrm{amp}}$ as a ``bridge''. Recalling the commentary preceding Lemma \ref{lemma_spectral1} in Section \ref{s:stl}, let 
$$
\Omega_- := \{\Re{z}\leq -1/2\}\cap\mathbb T
$$
be a closed sub-arc of the closed left semicircle whose boundary is separated from the boundary of $\{\Re z>0\}\cap \mathbb T$. Then let $\eta_-\colon\mathbb T\to [0,1]$ be a smooth function, with $\eta|_{\Omega_- } = 1$ and vanishing outside a (say) $1/10$-neighborhood of $\Omega_-.$ 
 
For the corresponding spectral projection $1_{\Omega_-}(u'_{\mathrm{amp}})$, of $u'_{\mathrm{amp}}$, since 
\[\eta_-(u\oplus u'_{\mathrm{path}})r = r\eta_-(u\oplus u'_{\mathrm{path}}) = 0,\]
compute 
\begin{multline*}
\|1_{\Omega_-}(u'_{\mathrm{amp}})r\|\leq \|\eta_-(u'_{\mathrm{amp}})r\|\leq \| \eta_-(u\oplus u'_{\mathrm{path}})r - \eta_-(u'_{\mathrm{amp}})r \|\leq \\ \leq C\|u\oplus u'_{\mathrm{path}} - u'_{\mathrm{amp}}\|\leq C\Big(7\ep+C\ep^{-1/2}\delta^{1/2}\Big)=:\gamma,
\end{multline*}
where $C= \|\eta_-\|_{\ol(\mathbb{T}) }$ is an absolute constant. That is, the projections $1_{\Omega_-}(u'_{\mathrm{amp}})$ and $r$ are approximately orthogonal, with the error coming from the transition among the unitaries $u\oplus u'_{\mathrm{path}}$ and $u'_{\mathrm{amp}}.$ 

Assuming that $\gamma<1/100,$ apply Corollary \ref{cor_two_projections} to produce a unitary $\sigma\in \mathrm{A}_{2d}$ such that 
\bee
\label{eq_sigma_properties}
\sigma 1_{\Omega_-}(u'_{\mathrm{amp}})\sigma^*\leq \bone-r = p+q, \quad \l\|\sigma-\bone\r\|\leq 5\gamma.
\ene
So let 
\[u^\dagger := \sigma u'_{\mathrm{amp}}\sigma^*, \quad v^\dagger := \sigma v'_{\mathrm{amp}}\sigma^*,\]
and notice that, up to constants, $u^\dagger$ and $v^\dagger$ satisfy the same assumptions as $u'_{\mathrm{amp}}$ and $v'_{\mathrm{amp}},$ but gain the property that the spectral subspace of $u^{\dagger}$, associated to $\Omega_-$, is contained in the range of $p+q$. We denote
\bee
\label{eq_sminus_def}
s_- := 1_{\Omega_-}(u^\dagger) = \sigma 1_{\Omega_-}(u'_{\mathrm{amp}})\sigma^*=(p+q)s_-.
\ene

Since $p$ and $q$ commute with $u\oplus u'_{\mathrm{path}}$ and therefore almost commute with $u'_{\mathrm{amp}}$ and $u^{\dagger}$, we have the following estimates for each of the commutators, either from equation \eqref{eq_sigma_properties} or from their construction:
\bee
\label{eq_commutators}
\|[p+q,u^\dagger]\|\le C\gamma,\quad \|[p+q,v^\dagger]\|\le C\gamma,\quad 0\le s_-\le p+q.
\ene
Let us now apply Lemma \ref{lemma_quadratic} to the commuting pair $u^{\dagger}$, $v^\dagger$, the projection $p+q$ that almost commutes with each of them (playing the role of projection $p$ in the lemma), and projection $s_-\le p+q$ which reduces both of them. We thus obtain block diagonal unitaries $g$ and $h$, acting on $\ran (p+q)$ and commuting with $s_-$, such that
\bee
\label{eq_g_h_initial_properties}
\|g-(p+q)u^\dagger|_{p+q}\|\le C\gamma^2,\quad \|h-(p+q)v^\dagger|_{p+q}\|\le C\gamma^2,\quad \|[g,h]\|\le C\gamma^2.
\ene
Since $s_-$ reduces both $g$ and $h$, we can further split them into two pairs of unitaries $(g_+,h_+)$ and $(g_-,h_-)$ acting on $\ran (p+q-s_-)$ and $\ran s_-$, respectively:
$$
g_+:=(p+q-s_-)g|_{p+q-s_-},\quad h_+:=(p+q-s_-)h|_{p+q-s_-},\quad g_-:=s_-g|_{ s_-},\quad h_-:=s_-h|_{ s_-},
$$
and from \eqref{eq_g_h_initial_properties} we still have
$$
\|[g_{\pm},h_{\pm}]\|\le C\gamma^2.
$$
Moreover, from the definition \eqref{eq_sminus_def} of $s_-$ we have large spectral gaps for the associated blocks of $u^{\dagger}$:
\bee
\label{eq_sigma_gap}
\sigma (s_-u^{\dagger}|_{s_-})\subset \Omega_-, \quad \sigma ((\bone-s_-)u^{\dagger}|_{\bone-s_-})\subset \mathbb T\setminus\Omega_-,
\ene
with each gap being of diameter at least $1/5$. Claim $(iii)$ of Lemma \ref{lemma_quadratic} implies that $g_+$ and $g_-$ inherit gaps of diameters $1/5-10(\gamma^2)$ at the same locations. As a consequence, for $\gamma<1/1000$, we can state that $g_+$ and $g_-$ each have spectral gaps of diameter at least $1/10$.

We can now apply Lemma \ref{lemma_almost_commuting_gapped}, with $\rho=1/10$, to the pairs $(g_+,h_+)$ and $(g_-,h_-)$. Thus, we produce two commuting pairs
$$
[g_+',h_+']=0,\quad [g_-',h_-']=0,\quad \|g_+'-g_+\|+\|g_-'-g_-\|+\|h_+'-h_+\|+\|h_-'-h_-\|\le C\gamma.
$$
Then, $g'=g_+'+g_-'$ and $h'=r_+'+r_-'$ are unitaries on $\ran (p+q)$, satisfying
\bee
\label{eq_satisfying}
[g',h']=0, \quad \|g'-u\oplus u'_-\|+\|h'-v\oplus v'_-\|\le C\gamma,\quad \text{where}\quad v'_-:=q v'_{\mathrm{path}}|_{q},\quad u'_{-}=q u'_{\mathrm{path}}|_q.
\ene

\subsection{Proof of Theorem \ref{th_gap_opening_intro}: the second dimension reduction.} \label{ss:secondreduction} \hfill\newline 

We now find ourselves in a situation somewhat similar to the previous Subsection \ref{ss:firstreduction}. Similarly to disposing of $r$, our goal is now to dispose of $q$.

We start with commuting operators $g'$ and $h'$ satisfying \eqref{eq_satisfying}. By repeating the steps that led to equation \eqref{eq_sigma_properties}, using the opposite semi-circle $\{\Re z\le 0\}$ and the spectral projection $1_{\{\Re z>1/2\}}(g')$, we generate another unitary $\tau\in U(\ran(p+q))$ with
\bee
\label{eq_tau_close_1}
\|\tau-\bone_{\ran (p+q)}\|< 100 \gamma
\ene
and
\[
g^\dagger:=\tau g'\tau^*,\quad h^\dagger:=\tau h' \tau^*,\quad s_+:=\tau 1_{\{\Re z>1/2\}}(g')\tau^*=1_{\{\Re z>1/2\}}(g^\dagger)\leq p.
\] 
From \eqref{eq_satisfying} and \eqref{eq_tau_close_1}, we have 
\[
\|g^\dagger-u\oplus u'_-\|+\|h^\dagger-v\oplus v'_-\|\le C\gamma.
\]
Similarly to the previous subsection, the pair $(g^\dagger,h^{\dagger})$ can now replace the pair $(g',h')$ for all practical purposes. With this replacement, we gained an additional property that $1_{\{\Re z>1/2\}}(g^{\dagger})\le p$. The arguments similar to those leading to \eqref{eq_commutators} now provide
$$
\|[p,g^\dagger]\|\le C\gamma,\quad \|[p,h^\dagger]\|\le C\gamma,\quad 0\le s_+\le p.
$$
As a consequence, we can again apply Lemma \ref{lemma_quadratic} with the almost commuting projection $p$, reducing projection $s_+$, and both $s_+g^\dagger|_{s_+}$ and $(p-s_+)g^\dagger|_{p-s_+}$ having spectral gaps of diameters $1/20$ centered at $z=1$ and $z=-1$, respectively.

As a result, we obtain two pairs of $C\gamma$-almost commuting unitary operators, and the first element of each pair has a spectral gap of diameter, say, $1/30$. This allows another, final,  application of Lemma \ref{lemma_almost_commuting_gapped} to both pairs, resulting in two commuting pairs of unitaries $(g'_+, h'_+)$ on $\ran(s_+)$ and $(g_-', h'_-)$ on $\ran(p-s_+)$, which sum to the final commuting pair 
$$
u':=g'_++g'_-,\quad v':=h'_++h'_-
$$
with
\[\|u'-u\| + \|v'-v\|< C\gamma.\]
The commuting unitary matrices $u'$, $v'$, together with the estimate above, can now be used to conclude the proof of Theorem \ref{th_gap_opening_intro} by noting that 
$$
\varepsilon = \delta^{1/3},\quad \gamma = C\delta^{1/3}
$$
provides the balanced choice of the parameters leading to the distance estimate  $C\|[u,v]\|^{1/3}$.\qed

%\begin{tikzpicture}
%\draw [blue,thick,domain=180:300] plot ({cos(\x-60)}, {sin(\x-60)});
%\filldraw[black] (-1,0) circle (0pt) node[anchor=east]{$p+q\geq s_-$};
% \draw[dotted,|-|]
%   (0,1+10pt)
%   arc[start angle=120,end angle=240,radius=1]
%  node[midway,fill=white,sloped] {\tiny{$s_-$}};
%\draw [red,thick,domain=270:450] plot ({cos(\x)}, {sin(\x)});
%\filldraw[black] (1,0) circle (0pt) node[anchor=west]{$r\perp (p+q)$};
%\end{tikzpicture}

%\begin{tikzpicture}
%\draw [blue,thick,domain=-60:60] plot ({cos(\x)}, {sin(\x)});
%\filldraw[black] (1,0) circle (0pt) node[anchor=west]{$s_+\leq p$};
%\draw [red,thick,domain=90:270] plot ({cos(\x)}, {sin(\x)});
%\filldraw[black] (-1,0) circle (0pt) node[anchor=east]{$p\perp q$};
%\end{tikzpicture}

\appendix

\section{Operator Lipschitz functions} \label{appendix-A}
In this section, we will discuss the definition and some properties of operator Lipschitz functions. Let $\mathcal F\subset \mathbb C$ be a closed subset, and $f\colon \mathcal F\to \mathbb C$ be a continuous function. We say that $f\in \OL(\mathcal F)$ (that is, operator Lipschitz on $\mathcal F$) if, for some $C>0$, we have
\begin{equation}
\label{eq_ol_estimate}
\|f(A_1)-f(A_2)\|\le C\|A_1-A_2\|    
\end{equation}
for all bounded normal normal operators $A_1,A_2\in B(H)$ with $\sigma(A_j)\in \mathcal F$. The smallest possible value of $C$ in equation \eqref{eq_ol_estimate} will be denoted by $\|f\|_{\OL(\mathcal F)}$. It is easy to see that $\|\cdot\|_{\OL(\mathcal F)}$ is a seminorm on $\OL(\mathcal F)$, that vanishes only on constant functions. Most commonly, one considers $\mathcal F=\mathbb C,\mathbb{R},\mathbb{T}$, which corresponds to functions defined on  normals, self-adjoint, and unitary operators, respectively. We refer the reader to \cite{APPS, AP2,AP3} for a comprehensive review of the properties of operator Lipschitz functions, with proofs and references. In what follows, we provide a summary of the properties that are used in the present paper. We start from some basic facts.
\begin{itemize}
    \item[(OL1)]Every operator Lipschitz function is Lipschitz, with Lipschitz constant equal to $\|f\|_{\OL(\mathcal F)}$. The converse is not necessarily true: for example, $|\cdot|\notin \|f\|_{\OL(\mathbb R)}$.
    \item[(OL2)] $\OL(\mathcal F)/\mathbb C$ (that is, operator Lipschitz functions modulo constant functions) is a Banach space.
    \item[(OL3)]Let $C_b^2(\mathbb{C})$ be the space of  bounded continuously twice differentiable functions with bounded first and second partial derivatives. Then $C_b^2(\mathbb{C})\subset \OL(\mathbb{C})$, and $\|f\|_{\OL(\mathbb{C})}\le C\|f\|_{C_b^2(\mathbb{C})}$. Note that the same also holds for $\mathbb C$ replaced by $\mathcal{F}$, if one defines $C_b^2(\mathcal{F})$ to be the set of functions that admit an extension in $C_b^2(\mathbb{C})$.
    \item[(OL4)] A linear function $f(z)=az+b$ is operator Lipschitz on $\mathbb{C}$, even though it does not belong to $C_b^2(\mathbb{C})$.
    \item[(OL5)]Let $a,b\in \mathbb{C}$, $a\neq 0$. Define an affine transformation $t\colon \mathbb C\to \mathbb{C}$ by $t(z):=az+b$. The operator Lipschitz norm behaves in the same way (and with the same proof) as the Lipschitz norm under such transformations. That is,
    $$
    \|f\circ t\|_{\OL(t^{-1}(\mathcal F))}=|a|\|f\|_{\OL(\mathcal F)}.
    $$
    \item[(OL6)]Every operator Lipschitz function is {\it commutator Lipschitz}. That is,
    $$
    \|[f(N),B]\|\le \|f\|_{\OL(\mathcal F)}\|[N,B]\|
    $$
    for every normal operator $N$ with $\sigma(N)\subset \mathcal F$ and every $B$ that is a unitary or a bounded self-adjoint.
   \item[(OL7)]Let $f\colon \mathbb R\to \mathbb C$ be $1$-periodic. Then $f\in \OL(\mathbb R)$ if and only if $f(t)=g(e^{2\pi i t})$, for some $g\in \OL(\mathbb T)$. In other words, a $1$-periodic function is operator Lipschitz on $\mathbb R$ if and only if the associate function on the circle is operator Lipschitz on $\mathbb T$.
   \item[(OL8)]Suppose that $f,g\in \OL(\mathcal F)\cap L^{\infty}(\mathcal F)$ are bounded operator Lipschitz functions. Then $fg\in \OL(\mathcal F)$ and
   $$
   \|fg\|_{\OL(\mathcal F)}\le \|f\|_{L^{\infty}(\mathcal F)}\|g\|_{\OL(\mathcal F)}+\|f\|_{\OL(\mathcal F)}\|g\|_{L^{\infty}(\mathcal F)}.
   $$
\end{itemize}

Let $\mathbb{T}=\{z\in \mathbb C\colon |z|=1\}$ be the unit circle, and
$$
\mathbb{T}_{\rho}:=\{z\in \mathbb C\colon |z+1|\ge \rho\},\quad \rho>0,
$$
be the same circle with a small arc of distance $\rho$ from $-1$ removed. Denote by $\arg_-$ the branch of the complex argument function that maps $\bC\setminus (-\infty,0]$ continuously onto $(-\pi,\pi)$. The following estimate will be important in order to avoid additional commutator norm losses in the proof of the main result.
\begin{lem}
\label{lemma_arg_lipschitz}
There is an absolute constant $C>0$ such that, for  $0<\rho<2$, we have
$$
\|\arg_-\|_{\OL(\mathbb{T}_{\rho})}\le C \rho^{-1}.
$$
\end{lem}
\begin{proof}
Clearly, (OL3) implies a similar statement with $C\rho^{-2}$ in the right hand side. As a consequence, it is sufficient to prove the lemma for, say, $0<\rho<1/100$. Let $\arg_+\colon \bC\setminus[0,+\infty)\to (0,2\pi)$ be another branch of the complex argument that is smooth near $-1$. It is easy to see that
\bee
\label{eq_sign_argm}
\arg_-(z)=\arg_+(z)-\pi(1-\sign y),\quad z=x+iy,\quad |z+1|<1.
\ene
We will construct a smooth function on $\bC$ that coincides with $\arg_-$ on $\bT_{\rho}$ and whose operator Lipschitz norm we can control. In order to do so, let $s \colon \bR\to [-1,1]$ be a smooth function such that
$$
s(x)=\begin{cases}
-1,&x\le -1;\\
1,&x\ge 1.
\end{cases}
$$
Let also
\bee
\label{eq_s_bound}
s_{\rho}(z)=s_{\rho}(x+iy):=s(2y/\rho).
\ene
From (OL3) and (OL5), it follows that
$$
\|s_{\rho}\|_{\OL(\bC)}\le C\rho^{-1},
$$
where $C$ is an absolute constant. Let  $\varphi,\psi\colon \bC\to[0,1]$ be smooth compactly supported functions such that
$$
\psi(z)=\begin{cases}
1,&z\in \bT;\\
0,&\dist (z,\bT)\ge 1/10;
\end{cases}
\quad
\varphi(z)=\begin{cases}
1,&|z+1|\le 1/10;\\
0,&|z+1|\ge 1/5.
\end{cases}
$$
Finally, define
\bee
\arg_{\rho}(z):=\varphi(z)(\arg_+(z)+\pi(1-s_{\rho}(z)))+(1-\varphi(z))\psi(z)\arg_-(z).
\ene
The reader can easily check that $\arg_{\rho}(z)=\arg_-(z)$ for $z\in \bT_{\rho}$, since 
$$
s_\rho(z)=s_{\rho}(x+iy)=\sign(y),\quad z\in \bT_{\rho}\cap \supp\varphi.
$$
We have
$$
\|\arg_{\rho}\|_{\OL(\bC)}\le \|\varphi\arg_+-\pi\varphi\|_{\OL(\bC)}+\|(1-\varphi)\psi\arg_-\|_{\OL(\bC)}+\pi\|\varphi s_z\|_{\OL(\bC)}.
$$
Since $\arg_+$ is smooth on $\supp \varphi$ and $\arg_-$ is smooth on $\supp(1-\varphi)\psi$, Property (OL3) implies that the first two terms are bounded by absolute constants. The third term is bounded by $C'\rho^{-1}$ due to \eqref{eq_s_bound}, (OL8), and (OL3), where $C'$ is another absolute constant.
%\begin{rmk}
    
%\end{rmk}
\end{proof}
%\begin{prp}
%    Let $f: \mathbb{T} \rightarrow \mathbb{R}$ be a twice continuously differentiable function. Then $f$ is operator Lipschitz, and there is some constant $C'>0$ independent of $f$ so that
%$$
%\|f\|_{\OL(\mathbb{T})} \leq C' \cdot \max_{\alpha = 0,1,2}\sup_{\theta \in \mathbb{T}}\| %\frac{d^{\alpha} f}{d \theta^{\alpha}}(\theta)\|.
%$$
%\end{prp}


\begin{thebibliography}{99}
\bibitem{APPS}Aleksandrov~A., Peller~V., Potapov~D., Sukochev~F., 
{\it Functions of normal operators under perturbations}, 
Adv. Math. 226 (2011), No. 6, 5216 -- 5251.
\bibitem{AP2}Aleksandrov~A., Peller~V., 
{\it Estimates of operator moduli of continuity}, J. Funct. Anal. 261 (2011), No. 10, p. 2741--2796; 
arXiv:1104:3553.
\bibitem{AP3}Aleksandrov~A., Peller~V., {\it Operator and commutator moduli of continuity 
for normal operators}, Proceedings of the London Mathematical Society 105, 
No. 4 (2012), p. 821--851.

\bibitem{Berg}Berg~I., {\it On approximation of normal operators by weighted shifts}, Michigan Math. J. 21 (1975), no. 4, 377 -- 383.

\bibitem{BEEK} Bratteli~O., Elliott~G., Evans~D., Kishimoto~A., {\it Homotopy of a pair of approximately commuting unitaries in a simple $C^*$-algebra}, J. Func. Anal. 160 (1998), no. 2, 466 -- 523.

\bibitem{BDF}Brown~L., Douglas~R., Fillmore~P., {Unitary equivalence modulo the compact 
operators and extensions of $C^*$-algebras}, Proc. Conf. Operator Theory (Dalhousie 
Univ., Halifax, N. S. 1973), Lecture Notes in Mathematics 345, Springer, 1973, 
p. 58 -- 128.

\bibitem{Davidson}Davidson,~K, {\it Almost commuting Hermitian matrices}, Math. 
Scand. 56 (1985), p. 222--240.

\bibitem{DS} Davidson~K., Szarek~S., {\it Local operator theory, random matrices, and Banach spaces}, Handbook of the geometry of Banach spaces, Vol. I, 317366.
North-Holland Publishing Co., Amsterdam, 2001.

%\bibitem{Choi}Choi~M.~D., {\it Almost commuting matrices need not be nearly commuting}, Proc. Amer. Math. Soc. 102 (1988), 529--533.

\bibitem{Eff}Effros~E., {\it Dimensions and $C^*$-algebras}, CBMS Regional Conference Series in Mathematics, vol. 46, Conference Board of Mathematical Sciences, Washington, DC, 1981.

\bibitem{Eilers1}Eilers~S., Loring~T., Pedersen~G., {\it Morphisms of extensions of $C^*$-algebras: pushing forward the Busby invariant}, Adv. Math. 147 (1999), no. 1, 74 -- 109.

%\bibitem{FK1}
%Filonov~N., Kachkovskiy~I., {\it A Hilbert-Schmidt analog of Huaxin Lin's theorem}, arXiv:1008.4002.
%\bibitem{Shulman}Enders~D., Shulman~T., {\it Almost commuting matrices, cohomology, and dimension}, Ann. Sci. \'Ec. Norm. Sup\'er (4), 56 (2023), no. 6, 1653 -- 1683.

\bibitem{Exel_Loring}Exel~R., Loring~T., {\it Invariants of almost commuting unitaries}, J. Funct. Anal. 95 (1991), no. 2, 364 -- 376.

\bibitem{Exel}Exel~R., {\it The soft torus and applications to almost commuting matrices}, Pacific J. Math. 160 (1993), no. 2, 207 -- 217.

\bibitem{FR}Friis~P., R{\o}rdam~M., {\it Almost commuting self-adjoint matrices --- a shortproof of Huaxin Lin's theorem}, J. Reine Angew. Math. 479 (1996), 121--131.

\bibitem{FS}Filonov~N., Safarov~Y., {\it On the relation between the operator and its self-commutator}, J. Funct. Analysis 260 (2011), 2902 -- 2932.

\bibitem{Gong_Lin}Gong~G., Lin~H., {\it Almost multiplicative morphisms and almost commuting matrices}, J. Operator Theory 40 (1998), no. 2, 217 -- 275.



%\bibitem{Gl}Glebsky~L., {\it Almost commuting matrices with respect to normalized Hilbert-Schmidt norm}, arXiv:1002.3082 (2010).

%\bibitem{H1}Hadwin~D., Weihua~L., {\it A Note on Approximate Liftings}, arXiv:0804.1387v1 (2008).

\bibitem{Halmos}Halmos~P.~R., {\it Some unsolved problems of unknown depth about operators in Hilbert space}, Proc. Roy. Soc. Edinburgh Sect., A 76 (1976), 67--76.

\bibitem{H_orig} Hastings~M., {\it Making almost commuting matrices commute}, Comm. Math. Phys. 291 (2009), No. 2, p. 321--345.

\bibitem{HL}Hastings~M., Loring~T., {\it Almost commuting matrices, localized Wannier functions, and the quantum Hall effect}, J. Math. Phys. 51 (2010), no. 1, 015214.

\bibitem{Herrera}Herrera~D., {\it On Hastings' approach to Lin's Theorem for Almost Commuting Matrices}, preprint (2020), \url{https://arxiv.org/abs/2011.11800}.

\bibitem{Kato}Kato~T., {\it Perturbation Theory for Linear Operators, 2nd edition}, Classics in Mathematics, vol. 32, Springer, 1995
\bibitem{KS}Kachkovskiy~I., Safarov~Y., {\it Distance to normal elements in real rank zero $C^*$-algebras}, J. Amer. Math. Soc. 29 (2016), no. 1, 61 -- 80.

\bibitem{Lin_main}Lin~H., {\it Almost commuting selfadjoint matrices and applications}, in ``Operator Algebras and Their Applications'', Fields Inst. Commun. 13 (1997), p. 193 -- 233.

\bibitem{Lin_unitary}Lin~H., {\it Almost commuting unitaries and classification of purely infinite simple $C^*$-algebras}, J. Funct. Anal. 155 (1998), 1 -- 24.

%\bibitem{Lb}Lin~H., {An introduction to the classification of amenable $C^*$-algebras}, World Scientific, 2001.

\bibitem{Lin_exprank}Lin~H., {\it Exponential rank of $C^*$-algebras with real rank zero and Brown--Pedersen's conjecture}, J. Funct. Anal. 114 (1993), p. 1--11.

\bibitem{LorHas}Loring~T., Hastings~M., {\it Disordered topological insulators via $C^*$-algebras}, A Letter Journal Exploring Frontiers in Physics (EPL), no. 92 (2010), 67004.

\bibitem{LorSor}Loring~T., S{\o}rensen, {\it Almost Commuting Unitary Matrices Related
to Time Reversal}, Commun. Math. Phys. 323 (2013), 859 -- 887.

\bibitem{Osborne}Osborne~T., {\it Almost commuting unitaries with a spectral gap are near commuting unitaries}, Proc. Amer. Math. Soc. 137 (2009), no. 12, 4043 -- 4048.

\bibitem{Phillips}Phillips,~N.C., {\it Approximation by unitaries with finite spectrum in purely infinite $C^*$-algebras}, J. Funct. Anal. 120 (1994), no. 1, 98 -- 106.

\bibitem{Thiel_rr_multiplier}Thiel~H., {\it The real rank of some multiplier algebras}, preprint (2024), \url{https://arxiv.org/abs/2402.01022}.

\bibitem{Voiculescu} Voiculescu~D., {\it Asymptotically commuting finite rank unitary operators without commuting approximants}, Acta Sci. Math. 45 (1983), 429 -- 431

\end{thebibliography}
\end{document}